\documentclass[11pt]{amsart}

\usepackage[ansinew]{inputenc}
\usepackage[all]{xy}
\SelectTips{cm}{}
\usepackage[pagebackref,
  ,pdfborder={0 0 0}
  ,urlcolor=black,hypertexnames=false]{hyperref}
\hypersetup{pdfauthor={Caterina Campagnolo, Francesco Fournier-Facio, Yash Lodha and Marco Moraschini}
  ,pdftitle=Displacements techniques in bounded cohomology}

\usepackage{amsfonts}
\usepackage{amssymb, amsmath,amsthm, mathtools}
\usepackage{tabularx, hyperref}
\usepackage{enumerate}
\usepackage{mathrsfs}
\usepackage{todonotes}
\usepackage{bbm}

\usepackage{tikz}
\usepackage{tikz-cd}

\newtheorem{lemma}{Lemma}[section]
\newtheorem{thm}[lemma]{Theorem}
\newtheorem*{thm*}{Theorem}
\newtheorem{prop}[lemma]{Proposition}
\newtheorem*{prop*}{Proposition}
\newtheorem*{claim*}{Claim}
\newtheorem{cor}[lemma]{Corollary}

\theoremstyle{definition}
\newtheorem{defi}[lemma]{Definition}

\newtheorem*{quest*}{Question}
\newtheorem{probl}[lemma]{Problem}
\newtheorem{example}[lemma]{Example}

\newtheorem{rem}[lemma]{Remark}

\theoremstyle{definition}

\makeatletter
\newcommand\norm{\bBigg@{0.8}}

\makeatother
 
 \newcommand{\indnorm}[2][flex]{\csname #1l\endcsname\|#2%
                                 \csname #1r\endcsname\|\mathclose{}}
                                  \newcommand{\indnorml}[4][flex]{\csname #1l\endcsname\|#2%
                                 \csname #1r\endcsname\|_{#3}^{#4}\mathclose{}}

\newcommand{\ifsv}[2][norm]{\!\csname #1l\endcsname\bracevert\!#2\!%
                            \csname #1r\endcsname\bracevert\!}

\DeclareMathOperator{\comp}{comp}

\DeclareMathOperator{\id}{id}

\DeclareMathOperator{\GL}{GL}

\DeclareMathOperator{\Xsep}{\mathfrak{X}^{\textup{sep}}}

\DeclareMathOperator{\monod}{c. c. \Z \ c.}

\DeclareMathOperator{\vol}{\textup{vol}}
\DeclareMathOperator{\homeo}{\textup{Homeo}_c}

\DeclareMathOperator{\diffvol}{\textup{Diff}_{c, \vol}}

\DeclareMathOperator{\supp}{supp}

\DeclareMathOperator{\Wr}{\textup{Wr}}

\newcommand{\N}{\ensuremath {\mathbb{N}}}
\newcommand{\R} {\ensuremath {\mathbb{R}}}

\newcommand{\Z} {\ensuremath {\mathbb{Z}}}

\renewcommand{\rho}{\varrho}
\def\phi{\varphi}

\usepackage{color}
\usepackage{pdfcolmk}

\long\def\forget#1{}

\def\emptyset{\varnothing}

\makeatletter
\@namedef{subjclassname@2020}{%
  \textup{2020} Mathematics Subject Classification}
\makeatother

\begin{document}

\title[]{Displacement techniques in bounded cohomology}

\author[]{Caterina Campagnolo}
\address{Departamento de Matem\'aticas, Universidad Aut\'onoma de Madrid, Madrid, Spain}
\email{caterina.campagnolo@uam.es}

\author[]{Francesco Fournier-Facio}
\address{Department of Pure Mathematics and Mathematical Statistics, University of Cambridge, UK}
\email{ff373@cam.ac.uk}

\author[]{Yash Lodha}
\address{Department of Mathematics, University of Hawaii, USA}
\email{yashlodha763@gmail.com}

\author[]{Marco Moraschini}
\address{Dipartimento di Matematica, Universit{\`a} di Bologna, 40126 Bologna, Italy}
\email{marco.moraschini2@unibo.it}

\thanks{}

\date{\today.\ \copyright{ C.~Campagnolo, F.~Fournier-Facio, Y.~Lodha and M.~Moraschini}.
}

\begin{abstract}
Several algebraic criteria, reflecting displacement properties of transformation groups, have been used in the past years to prove vanishing of bounded cohomology and stable commutator length.\
Recently, the authors introduced the property of \emph{commuting cyclic conjugates}, a new displacement technique that is widely applicable and provides vanishing of the bounded cohomology in all positive degrees and all dual separable coefficients.\ In this note we consider the most recent along with the by now classical displacement techniques and we study implications among them as well as counterexamples.
\end{abstract}
\maketitle



\section{Introduction}
 
The study of the vanishing of \emph{bounded cohomology} $H_b^\bullet(-; -)$ has attracted a lot of interest from the very beginning of the theory, since it is a useful tool in geometric topology \cite{vbc} and rigidity theory \cite{Burger_Monod_2002}, and it characterises certain algebraic properties of groups.\ The first result in this direction is the celebrated characterization of amenable groups by Johnson~\cite{Johnson}: A group~$\Gamma$ is amenable if and only if its bounded cohomology is zero for all \emph{dual} Banach $\Gamma$-modules and all positive degrees.\ Similarly, one can characterise finite groups as the ones whose bounded cohomology is zero in all positive degrees and for \emph{all} Banach modules~\cite{Frigerio:book}.\ The vanishing of the bounded cohomology in low degrees and trivial real coefficients (i.e.\ endowed with the trivial action) has already strong algebraic consequences.\ For instance, via Bavard Duality~\cite{bavard} the \emph{stable commutator length} of a group $\Gamma$ is zero if $H_b^2(\Gamma; \R) = 0$.\ Finally, groups with trivial bounded cohomology in all positive degrees can be used to compute explicitly some full bounded cohomology rings~\cite{monod:thompson, fflm2, ccc}.\ This is the case for example of Thompson's group~$F$, whose bounded cohomology is trivial in positive degrees and real coefficients. This fact can be used to compute the bounded cohomology ring of Thompson's group $T$, leading to $H_b^\bullet(T; \R) = \R [x]$, where $x$ denotes the Euler class~\cite{monod:thompson, fflm2}.

Many vanishing results for the (bounded) cohomology available in the literature are based on \emph{displacement techniques}.\ Such approaches for ordinary (co)homology date back to works by Anderson~\cite{anderson} and Fisher~\cite{fisher1960group} in the context of homeomorphism groups, as well as to Mather's proof of the acyclicity of the compactly supported homeomorphisms group of $\R^n$, $H_{k \geq 1}(\homeo(\R^n); \Z) = 0$~\cite{mather1971vanishing}; and to Wagoner's work on classifying spaces in algebraic $K$-theory ~\cite{wagoner}.\ Based on the work of Mather, Baumslag--Dyer--Heller introduced a large class of acyclic groups called \emph{mitotic groups}~\cite{BDH}. Then Varadarajan~\cite{varadarajan} and Berrick~\cite{berrick} independently extended that class to the larger family of \emph{binate groups} (that includes \emph{dissipated groups} such as $\homeo(\R^n)$).\ In the bounded context the first displacement techniques appeared in the work by Matsumoto and Morita about the \emph{bounded acyclicity} of the group $\homeo(\R^n)$~\cite{Matsu-Mor} (recall that a group is \emph{boundedly acyclic} if its bounded cohomology is zero in all positive degrees and trivial real coefficients).\ Later Kotschick introduced the notion of \emph{commuting conjugates}~\cite{kotschick, fflodha}, a displacement property of a group $\Gamma$ that readily implies the vanishing of the stable commutator length of $\Gamma$.\ Kotschick's result has been recently extended by showing that the property of having commuting conjugates actually implies that the second bounded cohomology group with trivial real coefficients is zero~\cite{fflodha}.\ Moreover, approaches close to the one of Matsumoto and Morita also show that binate groups are boundedly acyclic~\cite{Loeh:dim, fflm2}.\ Finally, some new displacement techniques have been recently introduced in the study of $\Xsep$-\emph{boundedly acyclic groups}, i.e.\ groups with zero bounded cohomology in all positive degrees and all \emph{separable} dual modules~\cite{monod:thompson, ccc}.\ These new dynamical properties are called \emph{commuting cyclic conjugates}, or \emph{commuting $\Z$-conjugates} in some special cases~\cite{ccc}.\ A similar dynamical property appears in the work of Monod~\cite{monod:thompson}, which we refer to as having \emph{conjugates in commuting $\Z$-conjugate}.

The goal of this note is to compare all the different notions that have been introduced in the literature so far and explain which kind of vanishing result they provide.\ The largest portion of this paper is devoted to the construction of several examples of groups satisfying a certain combination of properties.\ In order to avoid cluttering, we group the comparison into three diagrams involving properties that are closely related: see Figures \ref{fig:vd-ccc}, \ref{fig:vd-czc} and \ref{fig:vd-binate}.\ The main result of this paper is that each of these diagrams is correct: The inclusions hold, and each of the regions formed by intersections is non-empty.\ The property of being mitotic is a bit of a special case, since it is an algebraic criterion that does not have a clear dynamical counterpart, so it will be treated separately.

\begin{figure}
\centering
\def\svgscale{0.6}
\begingroup%
  \makeatletter%
  \providecommand\color[2][]{%
    \errmessage{(Inkscape) Color is used for the text in Inkscape, but the package 'color.sty' is not loaded}%
    \renewcommand\color[2][]{}%
  }%
  \providecommand\transparent[1]{%
    \errmessage{(Inkscape) Transparency is used (non-zero) for the text in Inkscape, but the package 'transparent.sty' is not loaded}%
    \renewcommand\transparent[1]{}%
  }%
  \providecommand\rotatebox[2]{#2}%
  \newcommand*\fsize{\dimexpr\f@size pt\relax}%
  \newcommand*\lineheight[1]{\fontsize{\fsize}{#1\fsize}\selectfont}%
  \ifx\svgwidth\undefined%
    \setlength{\unitlength}{576bp}%
    \ifx\svgscale\undefined%
      \relax%
    \else%
      \setlength{\unitlength}{\unitlength * \real{\svgscale}}%
    \fi%
  \else%
    \setlength{\unitlength}{\svgwidth}%
  \fi%
  \global\let\svgwidth\undefined%
  \global\let\svgscale\undefined%
  \makeatother%
  \begin{picture}(1,0.5)%
    \lineheight{1}%
    \setlength\tabcolsep{0pt}%
    \put(0,0){\includegraphics[width=\unitlength,page=1]{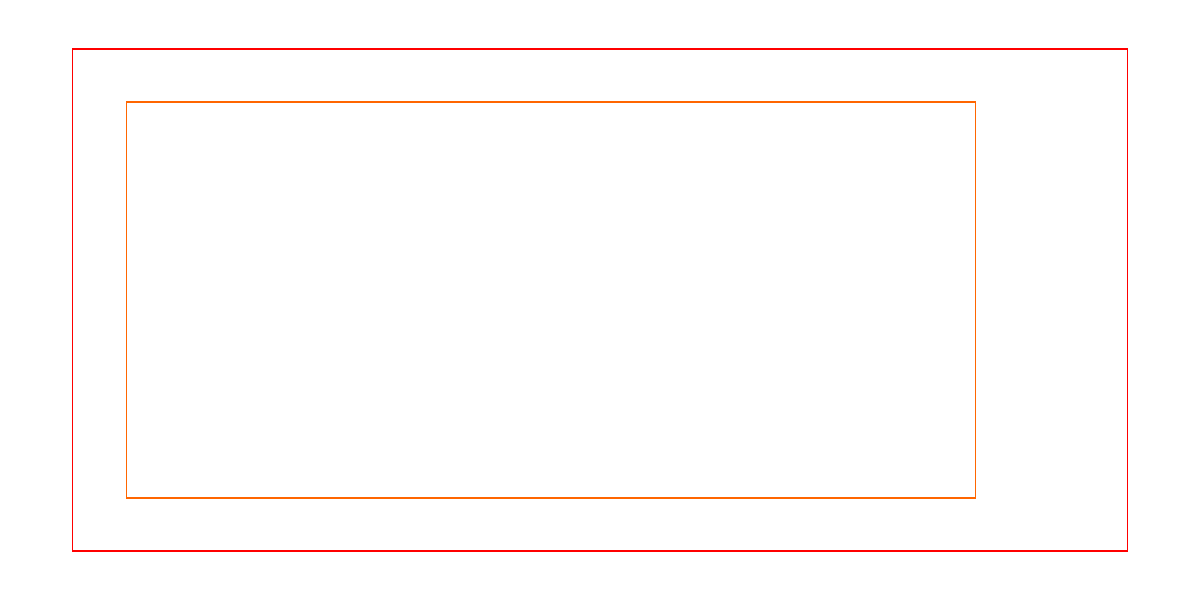}}%
    \put(0.12313004,0.39404196){\color[rgb]{1,0.4,0}\makebox(0,0)[lt]{\lineheight{1.25}\smash{\begin{tabular}[t]{l}c. c. c.\end{tabular}}}}%
    \put(0.07843316,0.43485366){\color[rgb]{1,0,0}\makebox(0,0)[lt]{\lineheight{1.25}\smash{\begin{tabular}[t]{l}c. c.\end{tabular}}}}%
    \put(0,0){\includegraphics[width=\unitlength,page=2]{vd-ccc.pdf}}%
    \put(0.15207397,0.35254685){\color[rgb]{0,0.50196078,0}\makebox(0,0)[lt]{\lineheight{1.25}\smash{\begin{tabular}[t]{l}c. $\Z/n$ c.\end{tabular}}}}%
    \put(0,0){\includegraphics[width=\unitlength,page=3]{vd-ccc.pdf}}%
    \put(0.60779304,0.37269681){\color[rgb]{0,0,1}\makebox(0,0)[lt]{\lineheight{1.25}\smash{\begin{tabular}[t]{l}c. $\Z$ c.\end{tabular}}}}%
  \end{picture}%
\endgroup%

\caption{Commuting conjugates, commuting cyclic conjugates, commuting $\Z$-conjugates and commuting $\Z/n$-conjugates.}
\label{fig:vd-ccc}
\end{figure}

\begin{figure}
\centering
\def\svgscale{0.6}
\begingroup%
  \makeatletter%
  \providecommand\color[2][]{%
    \errmessage{(Inkscape) Color is used for the text in Inkscape, but the package 'color.sty' is not loaded}%
    \renewcommand\color[2][]{}%
  }%
  \providecommand\transparent[1]{%
    \errmessage{(Inkscape) Transparency is used (non-zero) for the text in Inkscape, but the package 'transparent.sty' is not loaded}%
    \renewcommand\transparent[1]{}%
  }%
  \providecommand\rotatebox[2]{#2}%
  \newcommand*\fsize{\dimexpr\f@size pt\relax}%
  \newcommand*\lineheight[1]{\fontsize{\fsize}{#1\fsize}\selectfont}%
  \ifx\svgwidth\undefined%
    \setlength{\unitlength}{504bp}%
    \ifx\svgscale\undefined%
      \relax%
    \else%
      \setlength{\unitlength}{\unitlength * \real{\svgscale}}%
    \fi%
  \else%
    \setlength{\unitlength}{\svgwidth}%
  \fi%
  \global\let\svgwidth\undefined%
  \global\let\svgscale\undefined%
  \makeatother%
  \begin{picture}(1,0.57142857)%
    \lineheight{1}%
    \setlength\tabcolsep{0pt}%
    \put(0,0){\includegraphics[width=\unitlength,page=1]{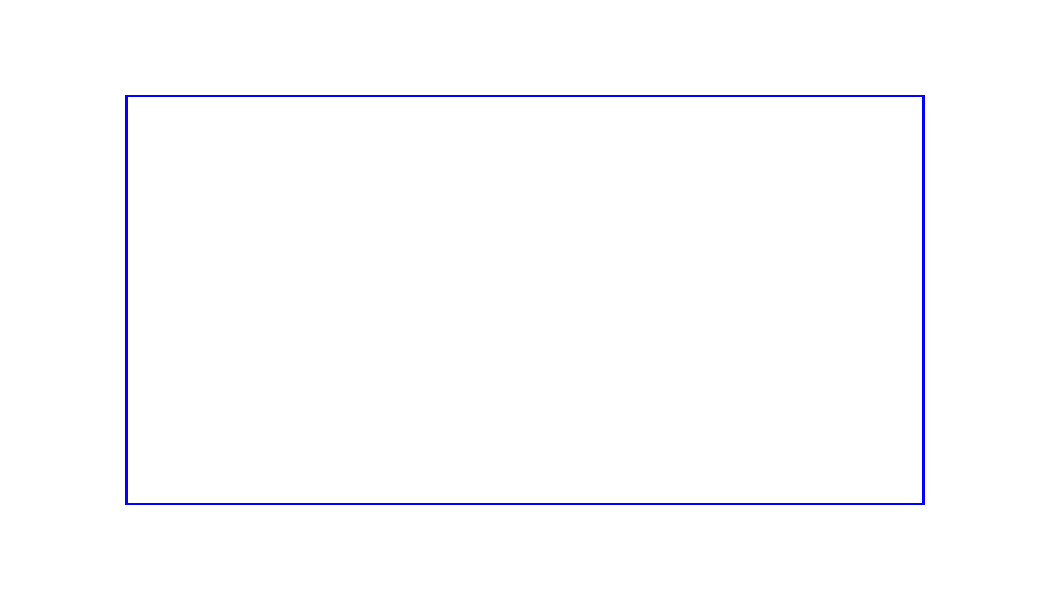}}%
    \put(0.14636679,0.44545478){\color[rgb]{0,0,1}\makebox(0,0)[lt]{\lineheight{1.25}\smash{\begin{tabular}[t]{l}c. $\Z$ c.\end{tabular}}}}%
    \put(0,0){\includegraphics[width=\unitlength,page=2]{vd-czc.pdf}}%
    \put(0.2207699,0.3731006){\color[rgb]{1,0,1}\makebox(0,0)[lt]{\lineheight{1.25}\smash{\begin{tabular}[t]{l}c. c. $\Z$ c.\end{tabular}}}}%
    \put(0,0){\includegraphics[width=\unitlength,page=3]{vd-czc.pdf}}%
    \put(0.68782584,0.32880007){\color[rgb]{0.47058824,0.26666667,0.12941176}\makebox(0,0)[lt]{\lineheight{1.25}\smash{\begin{tabular}[t]{l}dissipated\end{tabular}}}}%
  \end{picture}%
\endgroup%

\caption{Commuting $\Z$-conjugates, conjugates in commuting $\Z$-conjugate and dissipated.}
\label{fig:vd-czc}
\end{figure}

\begin{figure}
\centering
\def\svgscale{0.6}
\begingroup%
  \makeatletter%
  \providecommand\color[2][]{%
    \errmessage{(Inkscape) Color is used for the text in Inkscape, but the package 'color.sty' is not loaded}%
    \renewcommand\color[2][]{}%
  }%
  \providecommand\transparent[1]{%
    \errmessage{(Inkscape) Transparency is used (non-zero) for the text in Inkscape, but the package 'transparent.sty' is not loaded}%
    \renewcommand\transparent[1]{}%
  }%
  \providecommand\rotatebox[2]{#2}%
  \newcommand*\fsize{\dimexpr\f@size pt\relax}%
  \newcommand*\lineheight[1]{\fontsize{\fsize}{#1\fsize}\selectfont}%
  \ifx\svgwidth\undefined%
    \setlength{\unitlength}{576bp}%
    \ifx\svgscale\undefined%
      \relax%
    \else%
      \setlength{\unitlength}{\unitlength * \real{\svgscale}}%
    \fi%
  \else%
    \setlength{\unitlength}{\svgwidth}%
  \fi%
  \global\let\svgwidth\undefined%
  \global\let\svgscale\undefined%
  \makeatother%
  \begin{picture}(1,0.5)%
    \lineheight{1}%
    \setlength\tabcolsep{0pt}%
    \put(0,0){\includegraphics[width=\unitlength,page=1]{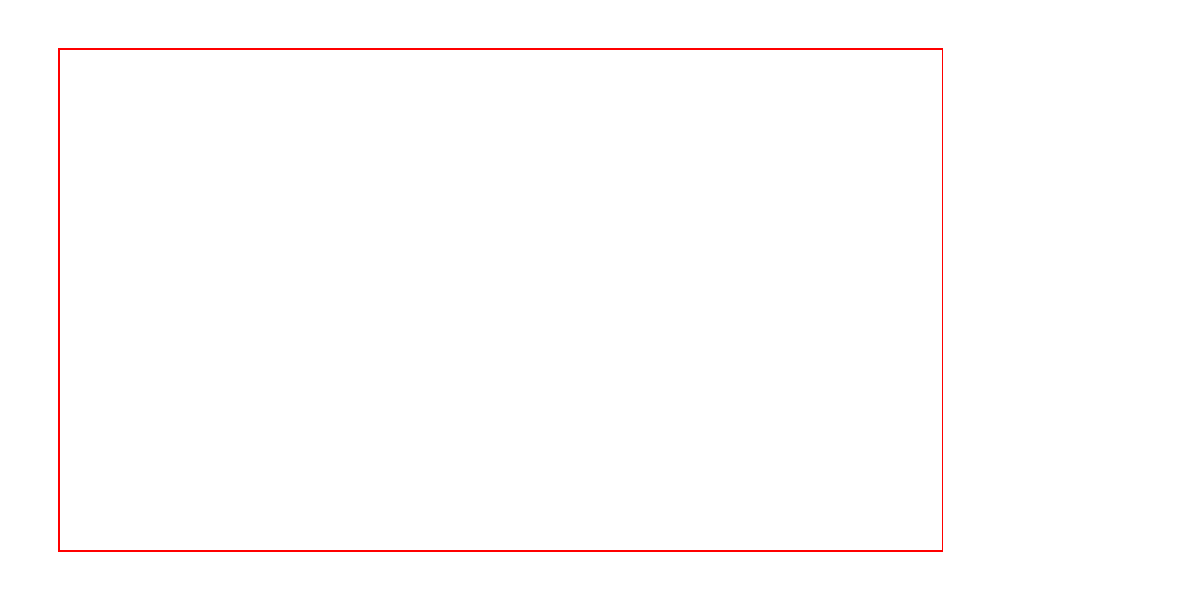}}%
    \put(0.06723723,0.43485366){\color[rgb]{1,0,0}\makebox(0,0)[lt]{\lineheight{1.25}\smash{\begin{tabular}[t]{l}c. c.\end{tabular}}}}%
    \put(0,0){\includegraphics[width=\unitlength,page=2]{vd-binate.pdf}}%
    \put(0.82276543,0.34593186){\color[rgb]{0,0,0.50196078}\makebox(0,0)[lt]{\lineheight{1.25}\smash{\begin{tabular}[t]{l}binate\end{tabular}}}}%
    \put(0,0){\includegraphics[width=\unitlength,page=3]{vd-binate.pdf}}%
    \put(0.58232439,0.27701827){\color[rgb]{0.47058824,0.26666667,0.12941176}\makebox(0,0)[lt]{\lineheight{1.25}\smash{\begin{tabular}[t]{l}dissipated\end{tabular}}}}%
  \end{picture}%
\endgroup%

\caption{Commuting conjugates, binate and dissipated.}
\label{fig:vd-binate}
\end{figure}

\newpage
\
\newpage
\subsection*{Plan of the paper} 
In Section~\ref{sec:def:impl} we recall the main definitions.\ More precisely, in Section~\ref{subsec:bc:scl} we recall the definition of bounded cohomology and stable commutator length.\ Then Section~\ref{subsec:bac} is devoted to the study of boundedly acyclic groups.\ Here we recall the notions of binate, mitotic and dissipated groups (Paragraph~\ref{subsubsec:binate:mitotic:dissipated}), commuting conjugates (Paragraph~\ref{subsubsec:cc:scl}), commuting cyclic conjugates, commuting $\Z/n$-conjugates and commuting $\Z$-conjugates (Paragraph~\ref{subsubsec:ccc}), as well as conjugates in commuting $\Z$-conjugate (Paragraph~\ref{subsubsec:ccZc}).

Section~\ref{sec:counterexample} contains the new results about the relations between the different displacement properties.\ In Section \ref{ss:vd-ccc} we show that Figure \ref{fig:vd-ccc} is correct.\ In particular, we will construct examples of groups with commuting $\Z/2$-conjugates that do not have commuting $\Z$-conjugates (Proposition \ref{prop:ccc:cznc}); however in order to stress how common this behaviour is, we provide more natural examples in Section \ref{ss:GL} (Proposition \ref{prop:GL}; see also Proposition \ref{prop:hall} in Section \ref{ss:vd-binate}).\ In Section \ref{ss:vd-czc} we show that Figure \ref{fig:vd-czc} is correct; and in Section \ref{ss:vd-binate} we show that Figure \ref{fig:vd-binate} is correct. In Section \ref{ss:mitotic} we discuss how mitotic groups fit in the picture.

Section~\ref{sec:problems} contains some open problems about displacement properties and bounded acyclicity with trivial or separable dual coefficients.

\subsection*{Acknowledgements}  
CC acknowledges support by the European Union in the form of a Mar\'ia Zambrano grant.\ FFF was supported by the Herchel Smith Postdoctoral Fellowship Fund.\ YL was partially supported by the NSF CAREER award 2240136.\ MM was partially supported by the INdAM--GNSAGA Project CUP E55F22000270001 and by the ERC ``Definable Algebraic Topology" DAT - Grant Agreement n. 101077154.

CC, FFF and YL thank also the Instituto de Ciencias Matem\'aticas (ICMAT) for its hospitality during the Thematic program on geometric group theory and low-dimensional geometry and topology (May-July 2023), where part of this research was conducted.\ MM also thanks the Forschungsinstitut f\"ur Mathematik (ETH Z\"urich) for its hospitality, where part of this research was conducted.

This work has been supported by the Madrid Government (Comunidad de Madrid - Spain) under the multiannual Agreement with UAM in the line for the Excellence of the University Research Staff in the context of the V PRICIT (Regional Programme of Research and Technological Innovation).\ This work has been funded by the European Union - NextGenerationEU under the National Recovery and Resilience Plan (PNRR) - Mission 4 Education and research - Component 2 From research to business - Investment 1.1 Notice Prin 2022 -  DD N. 104 del 2/2/2022, with title ``Geometry and topology of manifolds", proposal code 2022NMPLT8 - CUP J53D23003820001.

\section{Definitions and inclusions}\label{sec:def:impl}
In this note, $\Gamma$ will always denote an abstract discrete group. Some of the material recalled here is standard, and more detail can be found in e.g. \cite{monod, calegari_scl, Frigerio:book}.

\subsection{Bounded cohomology and stable commutator length}\label{subsec:bc:scl}

Given a group~$\Gamma$, a \emph{normed $\Gamma$-module} $E$ is a normed vector space endowed with a $\Gamma$-action by linear isometries.\ A \emph{Banach $\Gamma$-module} is a normed $\Gamma$-module whose underlying metric is complete.\ In this note the Banach $\Gamma$-module $\R$ is always assumed to carry the trivial $\Gamma$-action.\ Given a normed $\Gamma$-module $E$, its dual $E'$ carries naturally a contragredient $\Gamma$-action: $(g \cdot \lambda)(v) = \lambda(g^{-1}  v)$ for every $\lambda\in E', v\in E, g\in \Gamma$.\ This makes $E'$ into a normed $\Gamma$-module, which is always Banach.\ We call modules arising this way \emph{dual}.

Given a normed $\Gamma$-module $E$, the \emph{bounded cohomology of} $\Gamma$ \emph{with coefficients in} $E$, denoted by $H_b^\bullet(\Gamma; E)$, is defined as the cohomology of the bounded cochain complex $(\ell^\infty(\Gamma^{\bullet+1}; E)^\Gamma; \delta^\bullet)$, where $\delta^\bullet$ denotes the standard simplicial coboundary operator (which sends bounded functions to bounded functions) and $\ell^\infty(\Gamma^{\bullet+1}; E)^\Gamma$ denotes the subcomplex of invariant functions of $\ell^\infty(\Gamma^{\bullet+1}; E)$.

A classical result by Johnson shows that amenable groups can be characterised as those with trivial bounded cohomology in all positive degrees and all dual coefficients~\cite{Johnson}.\ On the other hand, the vanishing of the degree two bounded cohomology with trivial real coefficients is much more common than amenability, and provides useful information about the algebraic structure of the group.\ Indeed, given a group $\Gamma$ we can define the \emph{commutator length} of an element $g \in \Gamma' \coloneqq [\Gamma, \Gamma] \leq \Gamma$, denoted by $\textup{cl}(g)$, as the minimal integer $n \in \N$ such that $g$ can be written as the product of $n$ commutators.\ The \emph{stable commutator length} of an element $g \in \Gamma'$ is then defined as follows~\cite{calegari_scl}:
\[
\textup{scl}(g) = \lim_{n \to +\infty} \frac{\textup{cl}(g^n)}{n}.
\]
Bavard~\cite{bavard} showed that the the stable commutator length is identically zero on $\Gamma'$ if and only if the comparison map in degree two:
\[
\comp^2 \colon H^2_b(\Gamma; \R) \to H^2(\Gamma; \R),
\]
which is the map induced by the inclusion of the bounded cochain complex into the ordinary one, is injective.\ This result witnesses that groups with zero second bounded cohomology with real coefficients have zero stable commutator length.

The vanishing of second bounded cohomology with trivial real coefficients is more robust than the vanishing of stable commutator length. For instance, the fact that the vanishing of second bounded cohomology is stable under taking central extensions, while the vanishing of stable commutator length is not, is of crucial importance in the study of the spectrum of simplicial volume \cite{spectrum1, fflodha, spectrum2}. Moreover, if the second bounded cohomology with trivial real coefficients is zero, one can deduce strong consequences for the rigidity of circle actions \cite{matsumoto, ghys, burger:circle}.

\subsection{Boundedly acyclic groups}\label{subsec:bac}

In this note we are interested in boundedly acyclic groups~\cite{moraschini_raptis_2023, ccc}, i.e.\ groups with zero bounded cohomology in positive degrees:

\begin{defi}[Boundedly acyclic groups]
    A group $\Gamma$ is \emph{boundedly $n$-acyclic} if $H_b^k(\Gamma; \R) = 0$ for all integers $1 \leq k \leq n$.\ It is \emph{boundedly acyclic} if $n = \infty$.

    Similarly, a group $\Gamma$ is $\Xsep$-\emph{boundedly acyclic} if $H^k_b(\Gamma; E) = 0$ for all integers $k > 0$ and all separable dual Banach $\Gamma$-modules $E$.
\end{defi}

Recall that by Johnson's Theorem amenable groups are characterised by the vanishing of bounded cohomology with all dual coefficients. Therefore $\Xsep$-bounded acyclicity is in some sense the strongest vanishing property that one can hope for in non-amenable groups. Still, it has very strong consequences in rigidity theory, see \cite[Section 6]{ccc}.

In the sequel we describe some criteria for bounded acyclicity and $\Xsep$-bounded acyclicity.\ We also show some implications among the different vanishing results.

\subsubsection{Binate, mitotic and dissipated groups}\label{subsubsec:binate:mitotic:dissipated}
A wide class of acyclic groups are the \emph{binate} groups introduced independently by Berrick~\cite{berrick} and Varadarajan~\cite{varadarajan} (who called them \emph{pseudo-mitotic}).\ This class of groups includes groups arising from geometry and topology as well as from combinatorial group theory.\

\begin{defi}[Binate groups]
\label{def:binate}
A group $\Gamma$ is \emph{binate} if for every finitely generated subgroup $H \leq \Gamma$, there exists a homomorphism $f \colon H \to \Gamma$ such that $[H, f(H)] = 1$, and an element $t \in \Gamma$ such that ${}^t f(h) = h \cdot f(h)$ for all $h \in H$.
\end{defi}

\begin{rem}
    Here and in the sequel, we write ${}^t g \coloneqq tgt^{-1}$.
\end{rem}

Two special cases of binate groups have received special attention.\ The first class is the one of \emph{mitotic groups} introduced by Baumslag, Dyer and Heller~\cite{BDH}:

\begin{defi}[Mitotic groups]
\label{def:mitotic}
A group $\Gamma$ is \emph{mitotic} if for every finitely generated subgroup $H \leq \Gamma$ there exist $t_1, t_2 \in \Gamma$ such that $[H, {}^{t_1} H] = 1$ and ${}^{t_2} h = h \cdot {}^{t_1} h$ for all $h \in H$.
\end{defi}

In other words, the homomorphism $f \colon H \to \Gamma$ in the definition of binate groups is required to be a conjugate of the given embedding.

Mitotic groups are very suitable for constructions in combinatorial group theory.\ However, they do not appear naturally in geometry and dynamics (we will discuss this more in Section \ref{ss:mitotic}).\ A much more common class of binate groups, which contains the most natural examples of homeomorphism groups, is the class of \emph{dissipated groups}:

\begin{defi}[Boundedly supported groups]
\label{def:bsupp}
Let $\Gamma$ be a group of bijections of a set $X$, written as a directed union of a direct system of subsets $(X_i)_{i \in I}$, which we call \emph{bounded} sets. The group $\Gamma$ is \emph{boundedly supported} if every element $g \in \Gamma$ satisfies $\supp(g) \subset X_i$ for some $i \in I$.
\end{defi}

\begin{defi}[Dissipated groups]
\label{def:dissipated}

Let $X$ and $(X_i)_{i \in I}$ be as above, and let $\Gamma$ be a boundedly supported group of bijections of $X$.\ A \emph{dissipator} for $\Gamma$ and $X_i$ is an element $t_i \in \Gamma$ such that
\begin{enumerate}
    \item $t_i^p(X_i) \cap X_i = \emptyset$ for all $p \geq 1$;
    \item For every $g \in \Gamma$ supported on $X_i$, the element $f(g)$ defined as
    \[f(g) =
    \begin{cases}
    {}^{t_i^p} g &\text{ on } t_i^p(X_i), \text{ for every } p \geq 1; \\
    \id &\text{ elsewhere}
    \end{cases}\]
    belongs to $\Gamma$.
\end{enumerate}
The group $\Gamma$ is \emph{dissipated} if there exists a dissipator for $\Gamma$ and $X_i$ for every $i \in I$.
\end{defi}

\begin{rem}[Dissipated groups are binate]
\label{rem:diss}
    Dissipated groups are binate~\cite{berrick2002topologist}, as displayed in Figure \ref{fig:diss}.
\end{rem}

\begin{figure}
\centering
\def\svgscale{0.35}
\begingroup%
  \makeatletter%
  \providecommand\color[2][]{%
    \errmessage{(Inkscape) Color is used for the text in Inkscape, but the package 'color.sty' is not loaded}%
    \renewcommand\color[2][]{}%
  }%
  \providecommand\transparent[1]{%
    \errmessage{(Inkscape) Transparency is used (non-zero) for the text in Inkscape, but the package 'transparent.sty' is not loaded}%
    \renewcommand\transparent[1]{}%
  }%
  \providecommand\rotatebox[2]{#2}%
  \newcommand*\fsize{\dimexpr\f@size pt\relax}%
  \newcommand*\lineheight[1]{\fontsize{\fsize}{#1\fsize}\selectfont}%
  \ifx\svgwidth\undefined%
    \setlength{\unitlength}{1008bp}%
    \ifx\svgscale\undefined%
      \relax%
    \else%
      \setlength{\unitlength}{\unitlength * \real{\svgscale}}%
    \fi%
  \else%
    \setlength{\unitlength}{\svgwidth}%
  \fi%
  \global\let\svgwidth\undefined%
  \global\let\svgscale\undefined%
  \makeatother%
  \begin{picture}(1,0.28571429)%
    \lineheight{1}%
    \setlength\tabcolsep{0pt}%
    \put(0,0){\includegraphics[width=\unitlength,page=1]{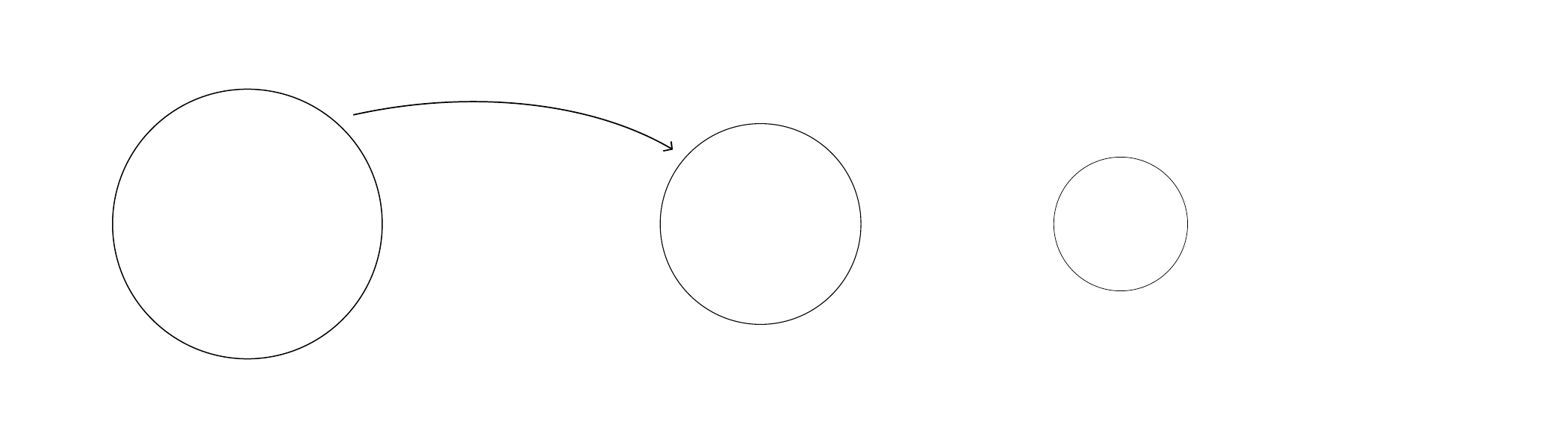}}%
    \put(0.14513931,0.13576107){\color[rgb]{0,0,0}\makebox(0,0)[lt]{\lineheight{1.25}\smash{\begin{tabular}[t]{l}$H$\end{tabular}}}}%
    \put(0.32201554,0.23028819){\color[rgb]{0,0,0}\makebox(0,0)[lt]{\lineheight{1.25}\smash{\begin{tabular}[t]{l}$t$\end{tabular}}}}%
    \put(0,0){\includegraphics[width=\unitlength,page=2]{diss.pdf}}%
    \put(0.60596112,0.20446034){\color[rgb]{0,0,0}\makebox(0,0)[lt]{\lineheight{1.25}\smash{\begin{tabular}[t]{l}$t$\end{tabular}}}}%
    \put(0,0){\includegraphics[width=\unitlength,page=3]{diss.pdf}}%
    \put(0.82539486,0.19073805){\color[rgb]{0,0,0}\makebox(0,0)[lt]{\lineheight{1.25}\smash{\begin{tabular}[t]{l}$t$\end{tabular}}}}%
    \put(0.89960353,0.13427297){\color[rgb]{0,0,0}\makebox(0,0)[lt]{\lineheight{1.25}\smash{\begin{tabular}[t]{l}$\cdots$\end{tabular}}}}%
    \put(0,0){\includegraphics[width=\unitlength,page=4]{diss.pdf}}%
    \put(0.6869273,0.01813105){\color[rgb]{0,0,0}\makebox(0,0)[lt]{\lineheight{1.25}\smash{\begin{tabular}[t]{l}$f(H)$\end{tabular}}}}%
  \end{picture}%
\endgroup%

\caption{\emph{Ad Remark \ref{rem:diss}:}\ $H$ is a finitely generated subgroup of $\Gamma$, supported on the ball on the left, and $t$ is an element all of whose powers displace this set, supported on the rectangle. Note that since $t$ is supported on the rectangle, the translates of the balls must shrink in size. The group $f(H)$ acts diagonally on all translates of the support of $H$.}
\label{fig:diss}
\end{figure}

Dissipated groups typically arise as \emph{groups of compactly supported homeomorphisms} \cite{binate:homeo}.\ Historically, the most important example in the context of bounded cohomology is the group of compactly supported homeomorphisms of $\R^n$:\ Indeed, based on the work by Mather in the case of ordinary (co)homology~\cite{mather1971vanishing}, Matsmumoto and Morita proved that the comparison map $\comp^k \colon H^k_b(\Gamma; \R) \to H^k(\Gamma; \R)$ is injective in all positive degrees~\cite{Matsu-Mor}.\ This was the  first example of a non-amenable boundedly acyclic group.

The same strategy then became available for binate groups: since they are acyclic \cite{berrick}, it is sufficient to show that their comparison map is injective in all positive degrees in order to conclude that binate groups are boundedly acyclic. This approach was first applied by L{\"o}h to the case of mitotic groups~\cite{Loeh:dim} and then more recently to all binate groups:

\begin{thm}[Binate groups are boundedly acyclic~\cite{fflm2}]
\label{thm:binate:bac}
Binate groups are boundedly acyclic.
\end{thm}

We have now covered the properties in Figure \ref{fig:vd-binate} and the corresponding inclusions, except for commuting conjugates, which is the subject of the next paragraph.

\subsubsection{Commuting conjugates, second bounded cohomology and stable commutator length}\label{subsubsec:cc:scl}

The next displacement property that appeared in the literature is the one of \emph{commuting conjugates}, which is the weakest and most natural one.\ This notion was introduced by Kotschick~\cite[Proposition 2.2]{kotschick} for studying the stable commutator length, and then deeply investigated in the context of second bounded cohomology~\cite{fflodha}.\ Kotschick's results are phrased with slightly different assumptions \cite[Proposition 2.1]{kotschick}, but the underlying property is the following:

\begin{defi}[Commuting conjugates]
\label{def:cc}
The group $\Gamma$ has \emph{commuting conjugates} (abbreviated c.\ c.) if for every finitely generated subgroup $H \leq \Gamma$ there exists $t \in \Gamma$ such that $[H, {}^t H] = 1$.
\end{defi}

As in the case of binate groups, this property arises very naturally in the case of boundedly supported groups.

\begin{example}[Commuting conjugates]
\label{ex:cc:action}
Let $X$ be a set with bounded subsets $(X_i)_{i \in I}$, and let $\Gamma$ be a boundedly supported group of bijections of $X$ such that for every $i \in I$ there exists $t_i \in \Gamma$ such that $X_i \cap t_i(X_i) = \emptyset$.\ Then $\Gamma$ has commuting conjugates, as displayed in Figure \ref{fig:cc}.
\end{example}

\begin{figure}
\centering
\def\svgscale{0.4}
\begingroup%
  \makeatletter%
  \providecommand\color[2][]{%
    \errmessage{(Inkscape) Color is used for the text in Inkscape, but the package 'color.sty' is not loaded}%
    \renewcommand\color[2][]{}%
  }%
  \providecommand\transparent[1]{%
    \errmessage{(Inkscape) Transparency is used (non-zero) for the text in Inkscape, but the package 'transparent.sty' is not loaded}%
    \renewcommand\transparent[1]{}%
  }%
  \providecommand\rotatebox[2]{#2}%
  \newcommand*\fsize{\dimexpr\f@size pt\relax}%
  \newcommand*\lineheight[1]{\fontsize{\fsize}{#1\fsize}\selectfont}%
  \ifx\svgwidth\undefined%
    \setlength{\unitlength}{648bp}%
    \ifx\svgscale\undefined%
      \relax%
    \else%
      \setlength{\unitlength}{\unitlength * \real{\svgscale}}%
    \fi%
  \else%
    \setlength{\unitlength}{\svgwidth}%
  \fi%
  \global\let\svgwidth\undefined%
  \global\let\svgscale\undefined%
  \makeatother%
  \begin{picture}(1,0.44444444)%
    \lineheight{1}%
    \setlength\tabcolsep{0pt}%
    \put(0,0){\includegraphics[width=\unitlength,page=1]{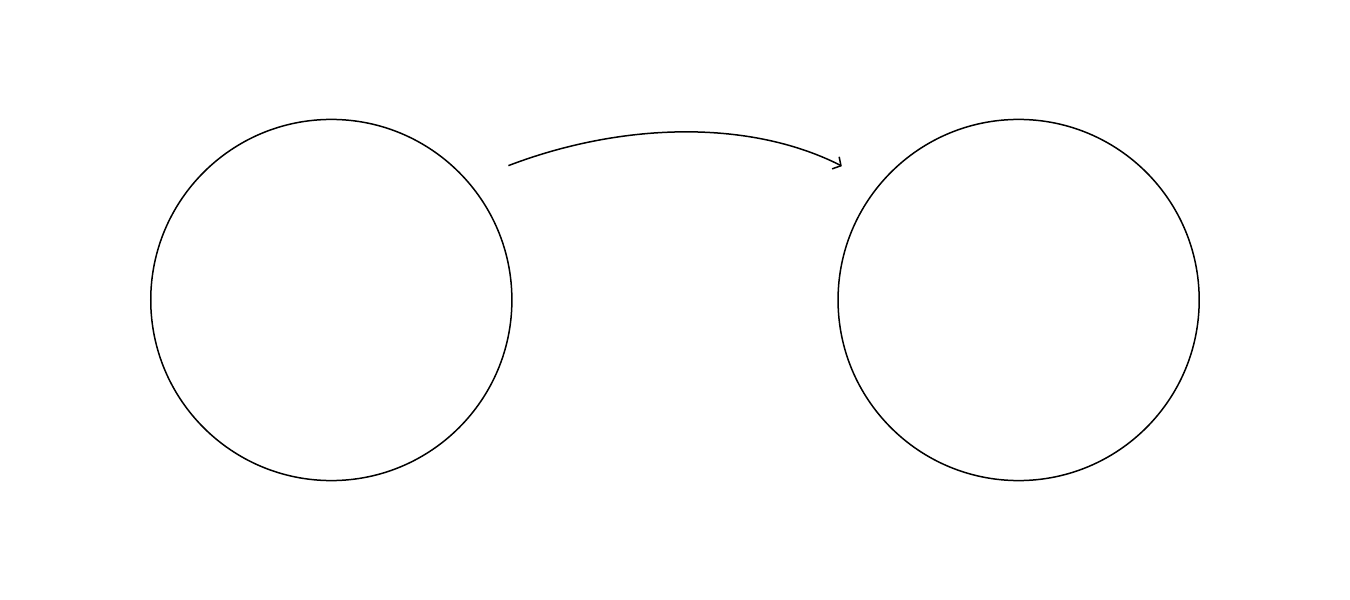}}%
    \put(0.22577226,0.21118388){\color[rgb]{0,0,0}\makebox(0,0)[lt]{\lineheight{1.25}\smash{\begin{tabular}[t]{l}$H$\end{tabular}}}}%
    \put(0.4819396,0.36202078){\color[rgb]{0,0,0}\makebox(0,0)[lt]{\lineheight{1.25}\smash{\begin{tabular}[t]{l}$t$\end{tabular}}}}%
    \put(0.72808711,0.2113474){\color[rgb]{0,0,0}\makebox(0,0)[lt]{\lineheight{1.25}\smash{\begin{tabular}[t]{l}${}^t H$\end{tabular}}}}%
    \put(0,0){\includegraphics[width=\unitlength,page=2]{cc.pdf}}%
  \end{picture}%
\endgroup%

\caption{\emph{Ad Example \ref{ex:cc:action}:}\ $H$ is a finitely generated subgroup of $\Gamma$, supported on the ball on the left, and $t$ is an element that displaces this set, supported on the rectangle.}
\label{fig:cc}
\end{figure}

It is immediate from the definition that mitotic groups have commuting conjugates.\ Moreover, Example \ref{ex:cc:action} shows that dissipated groups have commuting conjugates. This completes the inclusions in Figure \ref{fig:vd-binate}.\ We will discuss the missing examples in order to prove the correctness of Figure~\ref{fig:vd-binate} in Section~\ref{ss:vd-binate}.

\medskip

The simple property of commuting conjugates was shown by Kotschick to imply vanishing of stable commutator length:

\begin{thm}[Commuting conjugates implies vanishing of scl~\cite{kotschick}]
\label{thm:cc:scl}

Let $\Gamma$ be a group with commuting conjugates.\ Then the stable commutator length vanishes identically on $\Gamma'$.
\end{thm}

As we mentioned earlier, by Bavard Duality \cite{bavard} this result is equivalent to the injectivity of the comparison map in degree two \cite[Corollary 2.11]{Frigerio:book}.
This was then strengthened to the vanishing of the whole second bounded cohomology group:

\begin{thm}[Commuting conjugates implies vanishing of $H_b^2$~\cite{fflodha}]
\label{thm:cc:2bac}

Let $\Gamma$ be a group with commuting conjugates.\ Then $\Gamma$ is boundedly $2$-acyclic.
\end{thm}

\subsubsection{Commuting cyclic conjugates, commuting $\Z$-conjugates and commuting $\Z/n$-conjugates}\label{subsubsec:ccc}

We now move to the properties in Figure \ref{fig:vd-ccc}.
For a group to have commuting conjugates, there needs to be no extra structure on the sets of elements $t$ that appear in the definition.\ However, in order to obtain the vanishing of bounded cohomology in higher degrees, we have to require that the sets of conjugates are well organized ~\cite{ccc}. This leads to the definition of \emph{commuting cyclic conjugates}:

\begin{defi}[Commuting cyclic conjugates~\cite{ccc}]
\label{def:ccc}
The group $\Gamma$ has \emph{commuting cyclic conjugates} (abbreviated c.\ c.\ c.) if for every finitely generated subgroup $H \leq \Gamma$ there exist $t \in \Gamma$ and $n \in \Z_{\geq 2} \cup \{\infty\}$ such that $[H, {}^{t^p} H] = 1$ for $1 \leq p < n$ and $[H, t^n] = 1$.\ Here we read that $t^\infty = 1$.
\end{defi}

Let us formulate explicitly the naturally occurring case when $n$ is infinite:

\begin{defi}[Commuting $\Z$-conjugates~\cite{ccc}]
\label{def:czc}
The group $\Gamma$ has \emph{commuting $\mathbb{Z}$-conjugates} (abbreviated c.\ $\Z$ c.) if for every finitely generated subgroup $H \leq \Gamma$ there exists $t \in \Gamma$ such that $[H, {}^{t^p}H] = 1$ for all $p \geq 1$.
\end{defi}

\begin{example}[Commuting $\Z$-conjugates]\label{ex:czc:action}
Let $X$ be a set with bounded subsets $(X_i)_{i \in I}$, and let $\Gamma$ be a boundedly supported group of bijections of $X$ such that for every $i \in I$ there exists $t_i \in \Gamma$ such that $X_i \cap t_i^p(X_i) = \emptyset$ for all $p \geq 1$.\ Then $\Gamma$ has commuting $\mathbb{Z}$-conjugates, as displayed in Figure \ref{fig:czc}.
\end{example}

\begin{figure}
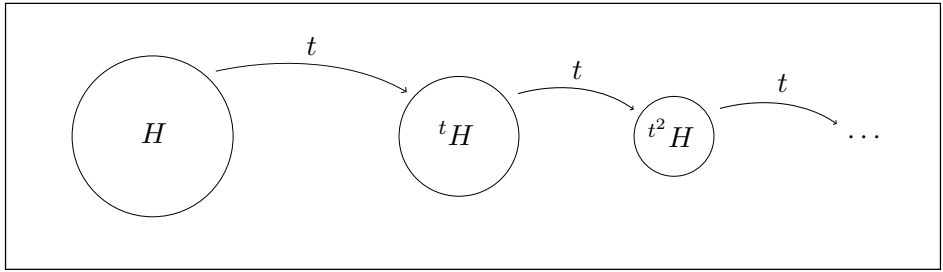

\centering
\def\svgscale{0.35}

\caption{\emph{Ad Example \ref{ex:czc:action}:}\ $H$ is a finitely generated subgroup of $\Gamma$, supported on the ball on the left, and $t$ is an element all of whose powers displace this set, supported on the rectangle. Note that since $t$ is supported on the rectangle, the translates of the balls must shrink in size.}
\label{fig:czc}
\end{figure}

The previous example shows that dissipated groups have commuting $\Z$-conjugates (whence, commuting cyclic conjugates).\ This explains one inclusion in Figure \ref{fig:vd-czc}.

\medskip

The main advantage of commuting $\Z$-conjugates over dissipated is that in the former case the existence of elements acting simultaneously on infinitely many components of support is not required. This recalls the Aristotelian distinction between actual and potential infinity \cite[p. 664]{monod:thompson}.\ As a specific instance of this, certain groups of diffeomorphisms or piecewise linear homeomorphisms have commuting $\Z$-conjugates, but they cannot be dissipated, because of the singularity that forms at the accumulation point of the shifted copies of bounded sets~\cite[p.~402]{kotschick}.

In turn, the advantage of commuting cyclic conjugates over commuting $\Z$-conjugates, is that one only needs elements that shift a bounded set a finite number of times, as long as the corresponding dynamics is well-behaved. This allows to encompass transformation groups that preserve an additional structure, which makes it impossible for a single element to shift sets infinitely many times; we will see a clear example of this in Section \ref{ss:GL}. We specialize the definition of commuting cyclic conjugates to describe the situations where finite orbits of shifted bounded sets suffice.

\begin{defi}[Commuting $\Z / n$-conjugates]
    \label{def:cnc}
    Let $n \geq 2$.\ The group $\Gamma$ has \emph{commuting $\Z/n$-conjugates} (abbreviated c.\ $\Z/n$ c.) if for every finitely generated subgroup $H \leq \Gamma$ there exists $t \in \Gamma$ such that $[H, {}^{t^p}H] = 1$ for $1 \leq p < n$, and $[H, t^n] = 1$.
\end{defi}

All the examples arising from geometry, combinatorial group theory, dynamics and algebraic topology that are known to have commuting cyclic conjugates are actually instances of either commuting $\Z$-conjugates or commuting $\Z/2$-conjugates~\cite{ccc}.\ On the other hand, it is possible to construct more complicated examples via combinatorial group theory (Section~\ref{ss:vd-ccc}).

\begin{example}[Commuting $\Z/2$-conjugates]
\label{ex:ccc:action}
Let $X$ be a set with bounded subsets $(X_i)_{i \in I}$, and let $\Gamma$ be a boundedly supported group of bijections of $X$ such that for every $i \in I$ there exists $t_i \in \Gamma$ such that $X_i \cap t_i(X_i) = \emptyset$ and $t_i^2$ fixes $X_i$ pointwise.\ Then $\Gamma$ has commuting $\Z/2$-conjugates, as displayed in Figure \ref{fig:c2c}.
\end{example}

\begin{figure}
\centering
\def\svgscale{0.4}
\begingroup%
  \makeatletter%
  \providecommand\color[2][]{%
    \errmessage{(Inkscape) Color is used for the text in Inkscape, but the package 'color.sty' is not loaded}%
    \renewcommand\color[2][]{}%
  }%
  \providecommand\transparent[1]{%
    \errmessage{(Inkscape) Transparency is used (non-zero) for the text in Inkscape, but the package 'transparent.sty' is not loaded}%
    \renewcommand\transparent[1]{}%
  }%
  \providecommand\rotatebox[2]{#2}%
  \newcommand*\fsize{\dimexpr\f@size pt\relax}%
  \newcommand*\lineheight[1]{\fontsize{\fsize}{#1\fsize}\selectfont}%
  \ifx\svgwidth\undefined%
    \setlength{\unitlength}{648bp}%
    \ifx\svgscale\undefined%
      \relax%
    \else%
      \setlength{\unitlength}{\unitlength * \real{\svgscale}}%
    \fi%
  \else%
    \setlength{\unitlength}{\svgwidth}%
  \fi%
  \global\let\svgwidth\undefined%
  \global\let\svgscale\undefined%
  \makeatother%
  \begin{picture}(1,0.44444444)%
    \lineheight{1}%
    \setlength\tabcolsep{0pt}%
    \put(0,0){\includegraphics[width=\unitlength,page=1]{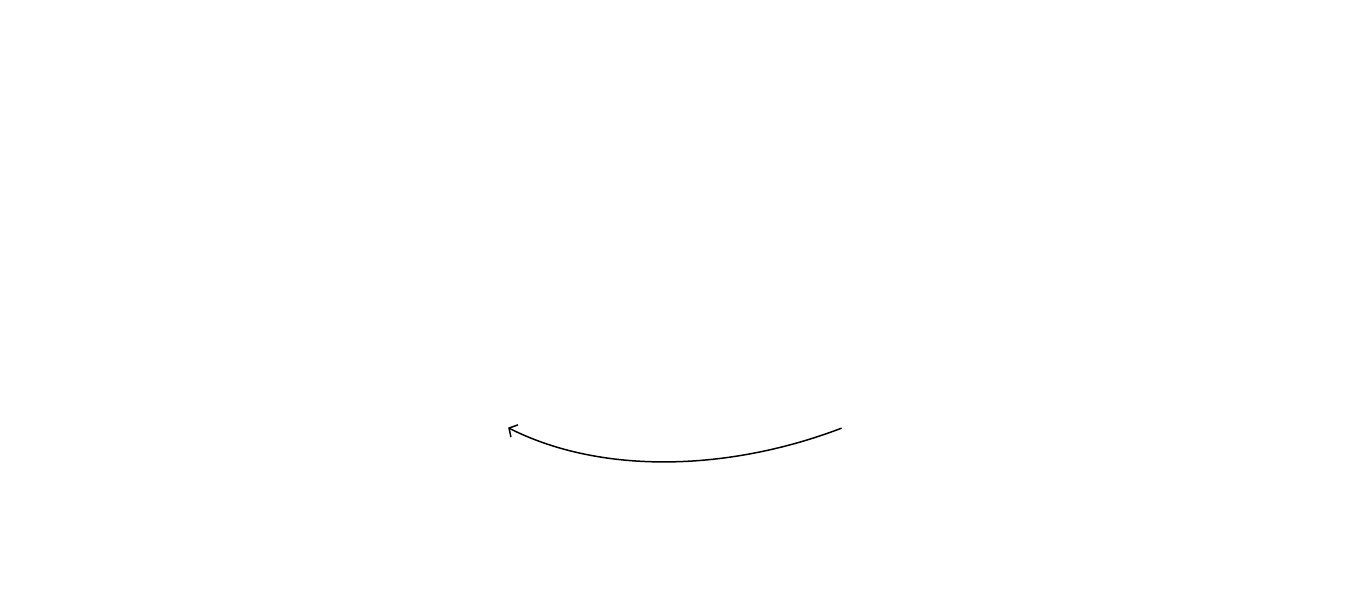}}%
    \put(0.48196008,0.05801876){\color[rgb]{0,0,0}\makebox(0,0)[lt]{\lineheight{1.25}\smash{\begin{tabular}[t]{l}$t$\end{tabular}}}}%
    \put(0,0){\includegraphics[width=\unitlength,page=2]{c2c.pdf}}%
    \put(0.22577226,0.21118388){\color[rgb]{0,0,0}\makebox(0,0)[lt]{\lineheight{1.25}\smash{\begin{tabular}[t]{l}$H$\end{tabular}}}}%
    \put(0,0){\includegraphics[width=\unitlength,page=3]{c2c.pdf}}%
    \put(0.4819396,0.36202078){\color[rgb]{0,0,0}\makebox(0,0)[lt]{\lineheight{1.25}\smash{\begin{tabular}[t]{l}$t$\end{tabular}}}}%
    \put(0.72808711,0.2113474){\color[rgb]{0,0,0}\makebox(0,0)[lt]{\lineheight{1.25}\smash{\begin{tabular}[t]{l}${}^t H$\end{tabular}}}}%
    \put(0,0){\includegraphics[width=\unitlength,page=4]{c2c.pdf}}%
  \end{picture}%
\endgroup%

\caption{\emph{Ad Example \ref{ex:ccc:action}:}\ $H$ is a finitely generated subgroup of $\Gamma$, supported on the ball on the left, and $t$ is an element that displaces this set, and whose square fixes the same set pointwise, supported on the rectangle.}
\label{fig:c2c}
\end{figure}

Recall that groups with commuting conjugates have zero second bounded cohomology (Theorem \ref{thm:cc:2bac}). As anticipated before, the notion of commuting cyclic conjugates, which slightly strengthens the notion of commuting conjugates, implies a much stronger vanishing result:

\begin{thm}[Commuting cyclic conjugates implies $\Xsep$-bounded acyclicity~\cite{ccc}]
\label{thm:ccc:bac}
Let $\Gamma$ be a group with commuting cyclic conjugates. Then $\Gamma$ is $\Xsep$-boundedly acyclic.
\end{thm}

This result already shows that the vanishing of the bounded cohomology in the case of dissipated groups goes beyond trivial real coefficients:

\begin{cor}[Dissipated groups are $\Xsep$-boundedly acyclic]
    Dissipated groups are $\Xsep$-boundedly acyclic.
\end{cor}

Note that it is not known whether all binate groups are $\Xsep$-boundedly acyclic (Problem \ref{probl:binate}).

\subsubsection{Conjugates in commuting $\Z$-conjugate}\label{subsubsec:ccZc}

A similar condition to commuting $\Z$-conjugates was previously introduced by Monod~\cite[Corollary~5]{monod:thompson}.\ For simplicity of notation, we give it a name here:

\begin{defi}[Conjugates in commuting $\Z$-conjugate]
\label{def:M}
The group $\Gamma$ has \emph{conjugates in commuting $\Z$-conjugate} (abbreviated $\monod$) if there exists a subgroup $\Lambda \leq \Gamma$ and an element $t \in \Gamma$ such that:
\begin{enumerate}
    \item $[\Lambda, {}^{t^p} \Lambda] = 1$ for all $p \geq 1$;
    \item Every finite subset of $\Gamma$ is contained in a conjugate of $\Lambda$.
\end{enumerate}
\end{defi}

\begin{example}[Conjugates in commuting $\Z$-conjugate]
\label{ex:M:action}
Let $X$ be a set with bounded subsets $(X_i)_{i \in I}$, and let $\Gamma$ be a boundedly supported group of bijections of $X$ such that there exists $i_0 \in I$ and $t_0 \in \Gamma$ such that:
\begin{enumerate}
    \item $X_{i_0} \cap t_0^p (X_{i_0}) =  \emptyset$ for all $p \geq 1$;
    \item For all $j$ there exists $s_j\in \Gamma$ such that $X_j \subset s_j (X_{i_0})$.
\end{enumerate}
Then $\Gamma$ has conjugates in commuting $\Z$-conjugate, as displayed in Figure \ref{fig:cczc}.
\end{example}

\begin{figure}
\centering
\def\svgscale{0.3}
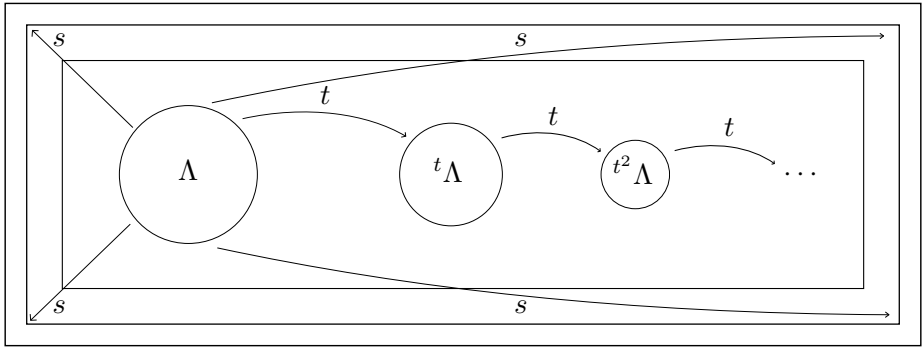
\caption{\emph{Ad Example \ref{ex:M:action}:}\ $\Lambda$ is a subgroup of $\Gamma$, supported on the ball on the left, and $t$ is an element that displaces this set, supported on the smallest rectangle. Given a finite subset of $\Gamma$, the union of the supports of its elements is contained in the middle rectangle, which is the image of the support of $\Lambda$ by an element $s$. The element $s$ is in turn supported on the largest rectangle.}
\label{fig:cczc}
\end{figure}

Also groups having conjugates in commuting $\Z$-conjugate have vanishing bounded cohomology for all separable coefficients:

\begin{thm}[$\monod$ implies $\Xsep$-bounded acyclicity~\cite{monod:thompson}]
\label{thm:M:bac}
Let $\Gamma$ be a group with conjugates in commuting $\Z$-conjugate.\ Then $\Gamma$ is $\Xsep$-boundedly acyclic.
\end{thm}

This result is in fact a consequence of the last unjustified inclusion in Figure \ref{fig:vd-czc}, which we prove here:

\begin{lemma}[$\monod$ implies c.\ $\Z$ c.]
\label{lem:M:implies:czc}
Let $\Gamma$ be a group with conjugates in commuting $\Z$-conjugate, then $\Gamma$ also has commuting $\mathbb{Z}$-conjugates.
\end{lemma}

\begin{proof}
Let $H \leq \Gamma$ be finitely generated.\ Then there exists $s \in \Gamma$ such that $H \leq {}^s \Lambda$.\ Let $t_0$ be as in the definition of $\monod$ and let $t := {}^{s} t_0$. Then for all $p \geq 1$ we have
\begin{align*}
    [H, {}^{t^p} H] &= [H, (s t_0^p s^{-1}) H (s t_0^{-p} s^{-1})] \\
    &\leq [s \Lambda s^{-1}, (s t_0^p s^{-1}) s \Lambda s^{-1} (s t_0^{-p} s^{-1})] \\
    &= s [\Lambda, {}^{t_0^p} \Lambda] s^{-1} = 1.\qedhere
\end{align*}
\end{proof}

This concludes the proof of all the inclusions in Figures \ref{fig:vd-ccc}, \ref{fig:vd-czc} and \ref{fig:vd-binate}.\ In the next section we will discuss the counterexamples witnessing strict inclusions, and we will exhibit groups lying in the intersections.

\section{Interactions between the properties}\label{sec:counterexample}

In the previous section, we have shown that all of the inclusions from the diagrams in Figures \ref{fig:vd-ccc}, \ref{fig:vd-czc} and \ref{fig:vd-binate} hold. In this section, we show that every non-empty area in the diagram does indeed contain a group.

\subsection{Correctness of Figure \ref{fig:vd-ccc}}
\label{ss:vd-ccc}

We will exhibit the following examples:

\begin{enumerate}
    \item A group with commuting conjugates that does not have commuting cyclic conjugates (Proposition \ref{prop:cc:not:ccc}).
    \item A group with commuting cyclic conjugates that does not have commuting $\Z/n$-conjugates, nor commuting $\Z$-conjugates (Proposition \ref{prop:ccc:cznc}).
    \item A group with commuting $\Z/n$-conjugates that does not have commuting $\Z$-conjugates (Proposition \ref{prop:ccc:cznc}).
    \item A group with commuting $\Z$-conjugates that does not have commuting $\Z/n$-conjugates for any $n$ (Proposition \ref{prop:czc:not:cznc}).
    \item A group which has commuting $\Z$-conjugates and commuting $\Z/n$-conjugates (Example \ref{ex:homeor2}).
\end{enumerate}

The last item is the most straightforward:

\begin{example}[c.\ $\Z$ c.\ and c.\ $\Z/n$ c.]
\label{ex:homeor2}
    Consider the group $\homeo(\R^2)$.\ It has commuting $\Z$-conjugates \cite[Corollary 5.19]{ccc}, as well as commuting $\Z/2$-conjugates \cite[Corollary 5.11 and its proof]{ccc}.\ It is straightforward to adapt the proof to show that it has commuting $\Z/n$-conjugates for every $n \geq 2$.
\end{example}

We proceed with a general construction that allows to treat the second and third items of the previous list, and moreover shows that the parameter $n$ in the definition of commuting $\Z/n$-conjugates does indeed make a difference.

We fix a sequence $\textbf{n} = (n_1, n_2, \ldots)$, where each $n_i \in \mathbb{Z}_{\geq 2}$.\ Set $\Gamma_0$ to be an arbitrary finitely generated non-abelian group, and define by induction $\Gamma_{i+1} \coloneqq \Gamma_i \wr \Z/n_{i+1}$.\
Recall that this denotes the standard restricted wreath product $\oplus_{n_{i+1}} \Gamma_i \rtimes \Z/n_{i+1}$.\
The standard embeddings $\Gamma_i \to \Gamma_{i+1}$ make this into a directed system of groups, and we denote by $\Gamma_{\mathbf{n}}$ their directed union.

\begin{lemma}
\label{lem:construction:cznc}
    For every prime $p$, the group $\Gamma_{\mathbf{n}}$ has commuting $\Z/p$-conjugates if and only if $p$ divides $n_i$ for infinitely many indices $i$.\ In all cases, $\Gamma_{\mathbf{n}}$ has commuting cyclic conjugates.
\end{lemma}

\begin{proof}
    First, suppose that $p$ divides $n_i$ for infinitely many indices $i$.\ Let $H \leq \Gamma$ be a finitely generated subgroup.\ Then $H \leq \Gamma_j$ for some $j$.\ By hypothesis there exists $i > j$ such that $p$ divides $n_i$, say $n_i = kp$ for some $1 \leq k < n_i$.\ We consider $\Gamma_i = \Gamma_{i-1} \wr \Z/n_i$, and let $t$ denote the generator of $\Z/n_i$.\ We claim that $t^k$ satisfies the conditions in the definition of commuting $\Z/p$-conjugates for $H$.\ First, for all $1 \leq q < p$ we have $[H, {}^{t^{kq}} H] \leq [\Gamma_{i-1}, {}^{t^{kq}} \Gamma_{i-1}] = 1$, since $1 \leq kq < kp=n_i$.\ Secondly, $t^{kp} = t^{n_i} = 1$, and so clearly $[H, t^{kp}] = 1$.\ This shows that $\Gamma_{\mathbf{n}}$ has commuting $\Z/p$-conjugates.

    Conversely, suppose that $\Gamma_{\mathbf{n}}$ has commuting $\Z/p$-conjugates.\ We will show that, for all $i$, there exists $j > i$ such that $p$ divides $n_j$.\ Since $\Gamma_i$ is finitely generated, there exists $t \in \Gamma_{\mathbf{n}}$ that satisfies the conditions in the definition of commuting $\Z/p$-conjugates for $\Gamma_i$.\ Let $j$ be such that $t \in \Gamma_j$; we pick $t$ so that $j$ is minimal. Since $\Gamma_i$ is not abelian, $t$ cannot belong to $\Gamma_i$, otherwise $[\Gamma_i, {}^t \Gamma_i] = [\Gamma_i, \Gamma_i] \neq 1$. This shows that $j > i$.\ We claim that $p$ divides $n_j$.

    We consider the group $\Gamma_j = \Gamma_{j-1} \wr \Z/n_j$, and notice that $\Gamma_i \leq \Gamma_{j-1}$.\ The element $t$ can be written uniquely as $t = s^k b$, where $b \in \Gamma_{j-1}^{n_j}$, $s$ is the generator of $\Z/n_j$ and $0 \leq k < n_j$.\ If $k = 0$, then the action of $t = b$ by conjugacy on $\Gamma_{j-1}$ is the same as the action of the $0$-th component of $b$, that we call $t'$.\ Since $\Gamma_i \leq \Gamma_{j-1}$, the element $t'$ satisfies the conditions in the definition of commuting $\Z/p$-conjugates for $\Gamma_i$, and $t' \in \Gamma_{j-1}$, which contradicts the choice of $t$ with minimal $j$.\
    It follows that $k \neq 0$.\ Now $t^p$ can be written uniquely as $s^{kp} b'$, where again $b' \in \Gamma_{j-1}^{n_j}$. If $s^{kp} \neq 1$, then ${}^{t^p} \Gamma_{j-1} = {}^{s^{kp}} \Gamma_{j-1}$, which is a different copy of $\Gamma_{j-1}$, and thus has trivial intersection with it.\ It follows that also ${}^{t^p} \Gamma_i$ has trivial intersection with $\Gamma_i$; therefore $[\Gamma_i, t^p] \neq 1$. So $s^{kp} = 1$, which implies that $p$ is not a unit modulo $n_j$, and since $p$ is prime, it must divide $n_j$.

    We have shown that, if $\Gamma_{\mathbf{n}}$ has commuting $\Z/p$-conjugates, then for all $i$ there exists $j > i$ such that $p$ divides $n_j$, i.e.\ $p$ divides infinitely many $n_j$, which concludes the proof.

    The fact that $\Gamma_{\mathbf{n}}$ has commuting cyclic conjugates follows directly from the construction: Every finitely generated subgroup $H$ is contained in some $\Gamma_i$, and we can choose $n = n_{i+1}$ and $t$ to be the generator of $\Z/n_{i+1} \leq \Gamma_{i+1}$.
\end{proof}

This allows to provide the following examples:

\begin{prop}[c.\ c.\ c.\ vs.\ c.\ $\Z\slash n$ c.\ and c.\ $\Z$ c.]
\label{prop:ccc:cznc}
    There exist groups with the following properties:
    \begin{itemize}
        \item For every non-empty set of primes $\mathcal{P}$ there exists a group that has commuting $\Z/p$-conjugates if and only if $p \in \mathcal{P}$;
        \item There exists a group that has commuting cyclic conjugates, but not commuting $\Z/n$-conjugates for any $n \geq 2$;
        \item Both cases can be chosen not to have commuting $\Z$-conjugates.
    \end{itemize}
\end{prop}

\begin{proof}
    We enumerate $\mathcal{P}$ as $p_1 < p_2 < \dots$, and set $n_i \coloneqq \prod_{j<i} p_j$.\ By Lemma \ref{lem:construction:cznc}, the resulting group $\Gamma_{\mathbf{n}}$ has commuting $\Z/p$-conjugates if and only if $p$ divides $n_i$ for infinitely many $i$, which happens exactly when $p = p_i$ for some $i$.

    If instead we take $n_i$ to be an increasing sequence of primes, then by Lemma \ref{lem:construction:cznc} the resulting group $\Gamma_{\mathbf{n}}$ does not have commuting $\Z/p$-conjugates for any prime $p$.\ Note that if a group has commuting $\Z/n$-conjugates, then it has commuting $\Z/m$-conjugates for every divisor $m$ of $n$, and so the more general statement follows.

    For the last item, we claim that if we start the construction with a finitely generated non-abelian \emph{torsion} group $\Gamma_0$, then the resulting groups $\Gamma_{\mathbf{n}}$ will not have commuting $\Z$-conjugates.\ By induction, each $\Gamma_i$ is torsion, and therefore so is $\Gamma_{\mathbf{n}}$.\ Suppose by contradiction that there exists $t \in \Gamma_{\mathbf{n}}$ such that $[\Gamma_0, {}^{t^p} \Gamma_0] = 1$ for all $p \geq 1$.\ Since $\Gamma_{\mathbf{n}}$ is torsion, taking $p$ to be the order of $t$ we obtain $[\Gamma_0, \Gamma_0] = 1$, which contradicts the choice of $\Gamma_0$ as non-abelian.
\end{proof}

\begin{rem}[Non-abelian torsion groups do not have c.\ $\Z$ c.]\label{rem:torsion:grps:no:cZc}
    The previous argument shows that in fact non-abelian torsion groups cannot have commuting $\Z$-conjugates.
\end{rem}

In particular, the previous proposition shows examples of groups with commuting $\Z/2$-conjugates that do not have commuting $\Z$-conjugates; we will see more naturally occurring examples in the next section. Next, we show:

\begin{prop}[c.\ $\Z$ c.\ not c.\ $\Z/n$ c.]
\label{prop:czc:not:cznc}
    The group $\homeo(\R)$ has commuting $\Z$-conjugates but does not have commuting $\Z/n$-conjugates for any $n \geq 2$.
\end{prop}

\begin{proof}
    The group $\homeo(\R)$ is dissipated and hence has commuting $\Z$-conjugates by Example \ref{ex:czc:action} (see also Example \ref{ex:dissipated:cczc} below). We show now that it does not have commuting $\Z/n$-conjugates for any $n \geq 2$.
    
    Let $H < \homeo(\R)$ be the copy of Thompson's group $F$ acting on $(0, 1)$ and fixing the rest of $\R$ pointwise.\ Note that $H$ is finitely generated and acts minimally on $(0, 1)$.\ If $H$ commutes with $t^n$, it follows in particular that $t^n(0, 1) = (0, 1)$. Since $t$ is compactly supported it preserves the orientation, and thus $t^n(0) = 0, t^n(1) = 1$.\  Similarly we deduce that $t(0) = 0$ and $t(1) = 1$.\ Thus ${}^{t} H$ is also supported on $(0, 1) = t(0, 1)$ and acts minimally on it.\ Let now $h \in H$ be an element whose only fixed point in $(0, 1)$ is $1/2$. Because ${}^{t} H$ commutes with $h$, it must preserve its fixed point set and thus fix $1/2$.\ This contradicts the minimality of the action of ${}^{t} H$.\ It follows that $H$ and ${}^{t} H$ do not commute.
\end{proof}

Note that this result is in striking contrast with the case of $\homeo(\R^2)$ discussed before in Example~\ref{ex:homeor2}.

Finally, we provide an example of a group with commuting conjugates that does not have commuting cyclic conjugates.

\begin{prop}[c.\ c.\ not c.\ c.\ c.]
\label{prop:cc:not:ccc}
    Let $\Gamma_1 \coloneqq \homeo(\R)$, and let $\Gamma_2$ be a group with commuting conjugates that does not have commuting $\Z$-conjugates. Then $\Gamma_1 \times \Gamma_2$ has commuting conjugates but does not have commuting cyclic conjugates.
\end{prop}

For example, $\Gamma_2$ may be chosen as one of the groups in Proposition \ref{prop:ccc:cznc}.

\begin{proof}
    The group $\Gamma \coloneqq \Gamma_1 \times \Gamma_2$ has commuting conjugates because each factor does.

    Now the proof of Proposition \ref{prop:czc:not:cznc} shows that there exists a finitely generated subgroup $H_1 < \Gamma_1$, such that if $[H_1, t_1^n] = 1$ for some $n \geq 2$ and some $t_1 \in \Gamma_1$, then $[H_1, {}^{t_1} H_1] \neq 1$.\ The assumption on $\Gamma_2$ implies that there exists a finitely generated subgroup $H_2 < \Gamma_2$ such that for all $t_2 \in \Gamma_2$ there exists $p \geq 1$ such that $[H_2, {}^{t_2^p} H_2] \neq 1$.

    Suppose by contradiction that $\Gamma$ has commuting cyclic conjugates, and let $H \coloneqq H_1 \times H_2$. Since $\Gamma$ has commuting cyclic conjugates, there exists $t = (t_1, t_2) \in \Gamma$ and $n \in \Z_{\geq 2} \cup \{ \infty \}$ such that $[H_i, {}^{t_i^p} H_i] = 1$ for all $1 \leq p < n$, and $[H_i, t_i^n] = 1$, for $i=1, 2$. If $n < \infty$, the statement with $i = 1$ contradicts the property of $H_1$; if $n = \infty$, the statement with $i = 2$ contradicts the property of $H_2$.
\end{proof}

\subsection{More examples with commuting $\Z/2$-conjugates but not commuting $\Z$-conjugates}
\label{ss:GL}

We will focus on an example of great interest in $K$-theory and homological stability:\ The infinite general linear group $\GL(\R)$.\
The group~$\GL(\R)$ has commuting cyclic conjugates (in fact commuting $\Z/2$-conjugates)~\cite[Corollary 5.45 and its proof]{ccc}. 
On the other hand, we are going to show that the infinite general linear group~$\GL(\R)$ cannot have commuting $\Z$-conjugates. The proof is based on the fact that the displacing element in $\GL(\R)$ has to \emph{conserve the dimension} (of the matrices in the finitely generated subgroup).\ For this reason, a similar argument does also apply to some other groups that \emph{preserve a quantity}~\cite{ccc}, including e.g.\
\begin{itemize}
    \item The group $\diffvol(\R^2)$ (\emph{conservation of volume});
    \item The stable mapping class group $\Gamma^1_\infty$ (\emph{conservation of genus}).
\end{itemize}
For linear groups note that our argument needs an appropriate notion of dimension or rank.\ For example, the general linear group of the cone of a ring is dissipated, and in particular it has commuting $\Z$-conjugates \cite[pp.~84-85]{berrick:approach}.

\begin{prop}[c.\ $\Z/2$ c.\ not c.\ $\Z$ c.]
\label{prop:GL}
The group $\GL(\R)$ has commuting $\Z/2$-conjugates but not commuting $\Z$-conjugates.
\end{prop}

\begin{proof}
The fact that $\GL(\R)$ has commuting $\Z/2$-conjugates has already been proved~\cite[Corollary 5.45 and its proof]{ccc}. We are left to show here that $\GL(\R)$ does not have commuting $\Z$-conjugates.

Let $\{ e_1, e_2, \ldots \}$ denote the standard basis elements of $\cup_{n\in\N} \R^n$, and let $H$ be the copy of $\GL_2(\Z)$ acting on $\R^2 \coloneqq \langle e_1, e_2 \rangle$ and fixing all other $e_i$.\ Note that $H \leq \GL(\R)$ is finitely generated.\
We compute via block multiplication the conjugation of an element of $\GL(\R)$ by an element of $H$:
\[
\begin{pmatrix}
    X & 0 \\
    0 & I
\end{pmatrix}
\begin{pmatrix}
    A & B \\
    C & D
\end{pmatrix}
\begin{pmatrix}
    X^{-1} & 0 \\
    0 & I
\end{pmatrix}
=
\begin{pmatrix}
    XAX^{-1} & XB \\
    CX^{-1} & D
\end{pmatrix}
,
\]
where $X \in \GL_2(\Z)$.\
Testing with both
\[
X = \begin{pmatrix}
    -1 & 0 \\
    0 & 1
\end{pmatrix}
\text{ and }
\begin{pmatrix}
    1 & 0 \\
    0 & -1
\end{pmatrix}
\]
we see that an element in the centralizer of $H$ must have $B = C = 0$ and $A$ diagonal.\ Moreover, testing with
\[X =
\begin{pmatrix}
    1 & 0 \\
    1 & 1
\end{pmatrix}\]
we see that $A$ must be a scalar multiple of the identity.\ It follows that for every $g$ in the centralizer of $H$, the action of $g$ on $\R^2 = \langle e_1, e_2 \rangle$ is by scalar multiplication.

Let $t \in \GL(\R)$ be such that $[H, {}^t H] = 1$.\ Then the action of ${}^t H$ on $\R^2$ is by scalar multiplication.\ For the same reason, the action of $H$ on $t\R^2$ is also by scalar multiplication.\  Since the action of $H$ on $\R^2$ has no one-dimensional invariant subspace, this shows that $\R^2 \cap t \R^2 = \{ 0 \}$.

Now suppose that there exists $t \in \GL(\R)$ such that $[H, {}^{t^p} H] = 1$ for all $p \geq 1$.\ Then the previous paragraph implies that $\R^2 \cap t^p \R^2 = \{ 0 \}$ for all $p \geq 1$.\ It follows that $t^{p_1} \R^2 \cap t^{p_2} \R^2 = \{ 0 \}$ for all $p_1 \neq p_2$.\ Therefore the span of $\{ t^{p} e_1, t^{p} e_2 \mid p \geq 1 \}$ is infinite-dimensional.\ But this is not possible since $t \in \GL_n(\R)$ for some $n$.\ Thus $\GL(\R)$ does not have commuting $\Z$-conjugates.
\end{proof}

\subsection{Correctness of Figure \ref{fig:vd-czc}}
\label{ss:vd-czc}

We will exhibit the following examples:

\begin{enumerate}
    \item A group with commuting $\Z$-conjugates that is not dissipated and does not have conjugates in commuting $\Z$-conjugate (Proposition \ref{prop:czc:not:diss:not:cczc}).
    \item A group with conjugates in commuting $\Z$-conjugate that is not dissipated (Example \ref{ex:F}).
    \item A dissipated group that does not have conjugates in commuting $\Z$-conjugate (Proposition \ref{prop:diss:not:cczc}).
    \item A group with conjugates in commuting $\Z$-conjugate that is also dissipated (Example \ref{ex:dissipated:cczc}).
\end{enumerate}

We start with two well-known examples:

\begin{example}[c.\ c.\ $\Z$ c. and dissipated]
\label{ex:dissipated:cczc}
    The group $\homeo(\R)$ is dissipated \cite{mather1971vanishing, berrick} and has conjugates in commuting $\Z$-conjugate \cite{monod:thompson, monodnariman}.
\end{example}

\begin{example}[c.\ c.\ $\Z$ c.\ not dissipated]
\label{ex:F}
    The derived subgroup $F'$ of Thompson's group $F$ \cite{cfp_96} has conjugates in commuting $\Z$-conjugate \cite{monod:thompson}. However it is not dissipated, in fact it is not binate. This is a combination of the computations of its homology~\cite{cohoT} (that is non trivial) and the fact that binate groups are acyclic~\cite{varadarajan, berrick}.\ In fact, it can be argued that $F'$ is not binate analogously to the proof of Proposition \ref{prop:czc:not:diss:not:cczc} below, without appealing to homological results.
\end{example}

We now move on to the two remaining examples, which will stem from the same construction. We construct two groups $\Gamma < G$ inside $\homeo(\R)$.\ We start with $I_1 \coloneqq (0, 1)$ and $\Gamma_1 = G_1$ the copy of Thompson's group $F$ (which is finitely generated) acting minimally on $I_1$ and fixing its complement.

Suppose by induction that we have defined a group $G_i$ supported on an interval $I_i$ and containing a finitely generated subgroup $\Gamma_i$.\ We choose an element $t_{i+1} \in \homeo(\R)$ with connected support, such that $t_{i+1}(I_i) \cap I_i = \emptyset$.\ For simplicity, we denote $I_i^p \coloneqq t_{i+1}^p(I_i)$ for all $p \in \Z$, and note that since $t_{i+1}$ is orientation-preserving, the intervals $I_i^p$ are pairwise disjoint.
We denote by $G_i^\omega$ the group of homeomorphisms $g$ supported on the disjoint union of intervals $\{ I_i^p \}_{p \in \Z}$ such that the restriction of $g$ on $I_i^p$ coincides with the action of an element of ${}^{t_{i+1}^p} G_i$.\ Note that we are allowing elements of $G_i^\omega$ to act simultaneously on infinitely many intervals, and therefore $G_i^\omega$ is naturally isomorphic to $\prod_{\Z} G_i$.\ We set $G_{i+1} \coloneqq \langle G_i^\omega, t_{i+1} \rangle$ which is naturally isomorphic to the unrestricted wreath product $G_i \Wr \Z$.\ We also set $\Gamma_{i+1} \coloneqq \langle \Gamma_i, t_{i+1} \rangle$, which is naturally isomorphic to $\Gamma_i \wr \Z$, seen as a subgroup of $G_i \wr \Z \leq G_i \Wr \Z = G_{i+1}$.\ Finally, we set $I_{i+1} \coloneqq \supp(t_{i+1}) = \supp(\Gamma_{i+1}) = \supp(G_{i+1})$.

Let $\Gamma < G < \homeo(\R)$ be the directed unions of the $\Gamma_i$ and the $G_i$ respectively.\ It follows from the construction that:

\begin{lemma}
\label{lem:mon:firstpart}
    The group $G$ is dissipated, and the subgroup $\Gamma < G$ has commuting $\Z$-conjugates.
\end{lemma}

The following lemma will be useful:

\begin{lemma}
\label{lem:normal:mon}
    For each $i \in \N$, the subgroup $\Gamma_i < G_i$ is normal.
\end{lemma}

\begin{proof}
    We proceed by induction on $i \in \N$, the case $i = 1$ being tautological.\ Suppose by induction that $\Gamma_i$ is normal in $G_i$.\ Consider $$\Gamma_{i+1} = \bigoplus_{\Z} \Gamma_i \rtimes \langle t_{i+1} \rangle \leq G_{i+1} = G_i^\omega \rtimes \langle t_{i+1} \rangle.$$ Note that it suffices to show that $\oplus_{\Z} \Gamma_i$ is normal in $G_i^\omega$ and invariant under the action of $t_{i+1}$.\ The latter is clear.\ Now a direct sum is normal in its corresponding direct product, and by induction hypothesis $\Gamma_i$ is normal in $G_i$.\ Thus
    \begin{align*}
        \bigoplus_{\Z} \Gamma_i&\unlhd \prod_{\Z}\Gamma_i\unlhd \prod_{\Z}G_i=G_i^\omega.\qedhere
    \end{align*}
\end{proof}

The key observation is the following:

\begin{lemma}
\label{lem:conjugacy:mon}
    For every $g \in G$ and every $i \in \N$, we have that either 
    \begin{itemize}
        \item $g(I_i) \cap I_i = \emptyset$, or 
        \item $g(I_i) = I_i$, ${}^g G_i = G_i$ and ${}^g \Gamma_i = \Gamma_i$.
    \end{itemize}
\end{lemma}

\begin{proof}
    We prove the claim by induction on $d = (j-i) \in \Z$ where $i$ is as in the statement, and $j$ is such that $g \in G_j$.\ If $j \leq i$, then $g$ is supported on $I_i$ and belongs to $G_i$, so $g(I_i) = I_i$ and ${}^g G_i = G_i$ are trivial, and ${}^g \Gamma_i = \Gamma_i$ follows from Lemma \ref{lem:normal:mon}.

    Now suppose that $g \in G_{i+1}$, i.e.\ $d = 1$.\ By construction $G_{i+1}$ is isomorphic to $G_i \Wr \Z$, with $t_{i+1}$ representing the stable letter.\ The copy of $G_i$ sitting at the lamp labelled by $k$ is supported on $I_i^k$ and fixes pointwise its complement, in particular it fixes pointwise $I_i^\ell$ for all $\ell \neq k$.\ We write $g = t_{i+1}^k b$ for some $b \in G_i^\omega$ and $k \in \Z$.\ If $k \neq 0$, then $g(I_i) = t_{i+1}^k b(I_i) = t_{i+1}^k (I_i) = I_i^k$ which is disjoint from $I_i$.\ Otherwise, $g = b$ preserves $I_i$, and the action of $g$ by conjugacy on $G_i$ is the same as the action of the $0$-th component of $g$, which belongs to $G_i$.\ Therefore ${}^g G_i = G_i$ and ${}^g \Gamma_i = \Gamma_i$, as in the first paragraph.

    Now suppose that $d > 1$, i.e.\ $j-1 > i$. Since $I_i \subset I_{j-1}$, if $g(I_{j-1}) \cap I_{j-1} = \emptyset$, then we are done.\ Otherwise, the above argument shows that $g(I_{j-1}) = I_{j-1}$, and there exists an element $h \in G_{j-1}$ such that $g|_{I_{j-1}} = h|_{I_{j-1}}$, and therefore the actions of $g$ and $h$ by conjugacy on $G_{j-1}$ coincide.\ Then we can apply the induction hypothesis to $h$ and conclude.
\end{proof}

We are now ready to provide one more example for Figure \ref{fig:vd-czc}:

\begin{prop}[Dissipated not c.\ c.\ $\Z$ c.]
\label{prop:diss:not:cczc}
    The group $G$ is dissipated, but does not have conjugates in commuting $\Z$-conjugate.
\end{prop}

\begin{proof}
    Let $\Gamma_i, G_i, t_i, \Gamma$ and $G$ be as in the construction above.\ We have already seen in Lemma \ref{lem:mon:firstpart} that $G$ is dissipated.

    Now suppose by contradiction that $G$ has conjugates in commuting $\Z$-conjugate.\ So there exist $\Lambda \leq G$ and $t \in G$ with the following properties:
    \begin{enumerate}
        \item $[\Lambda, {}^{t^p} \Lambda] = 1$ for all $p \geq 1$;
        \item Every finite subset of $G$ is conjugate into $\Lambda$.
    \end{enumerate}
    Let $n \in \N$ be such that $t \in G_n$.\ The second property implies that there exists $s \in G$ such that ${}^s \Gamma_n \leq \Lambda$, since $\Gamma_n$ is finitely generated.
    By Lemma \ref{lem:conjugacy:mon}, we split into two cases.\
    First, suppose that $s(I_n) \cap I_n = \emptyset$.\ Then $t$ commutes with ${}^s \Gamma_n$ because they are supported on disjoint intervals, and so $[{}^s \Gamma_n, {}^t ({}^s \Gamma_n)] = [{}^s \Gamma_n, {}^s \Gamma_n] \neq 1$, because $\Gamma_n$ is not abelian.\ In the other case, $s(I_n) = I_n$ and ${}^s \Gamma_n = \Gamma_n$ and so again $[{}^s \Gamma_n, {}^t ({}^s \Gamma_n)] = [\Gamma_n, {}^t \Gamma_n] = [\Gamma_n, \Gamma_n ]\neq 1$, where the last equality follows from Lemma \ref{lem:normal:mon}.\
    In both cases, $[{}^s \Gamma_n, {}^t ({}^s \Gamma_n)] \neq 1$ and thus $[\Lambda, {}^t \Lambda] \neq 1$, which contradicts the defining property of $t$.
\end{proof}

Finally, we will show that $\Gamma$ fits the last remaining non-empty region in Figure \ref{fig:vd-czc}: It has commuting $\Z$-conjugates, but it is neither dissipated nor it has conjugates in commuting $\Z$-conjugate.

\begin{lemma}
\label{lem:support:gamma}
    Let $g \in \Gamma$. Then the support of $g$ is a disjoint union of finitely many intervals.
\end{lemma}

\begin{proof}
    We prove this by induction on the index $i \in \N$ such that $g \in \Gamma_i$. If $i = 1$, then $g$ is an element of Thompson's group $F$ acting on the interval, and so the statement is well-known \cite{cfp_96}. Now suppose that this is true up to $\Gamma_i$, and let $g \in \Gamma_{i+1}$. Then we can write $g = t_{i+1}^k b$ for some $b \in \oplus_{\Z} \Gamma_i$ and some $k \in \Z$. If $k \neq 0$, then $\supp(g) = \supp(t_{i+1}) = I_{i+1}$, which is an interval.\ Otherwise, $g = b$ is the product of finitely many elements in conjugates of $\Gamma_i$ supported on disjoint sets, and so the induction hypothesis applies.
\end{proof}

\begin{prop}[c.\ $\Z$ c.\ not dissipated and not c.\ c.\ $\Z$ c.]
\label{prop:czc:not:diss:not:cczc}
    The group $\Gamma$ has commuting $\Z$-conjugates, but is not dissipated (in fact, it is not even binate) and does not have conjugates in commuting $\Z$-conjugate.
\end{prop}

\begin{proof}
    We have already seen in Lemma \ref{lem:mon:firstpart} that $\Gamma$ has commuting $\Z$-conjugates. It follows from (a shorter version of) the argument of Proposition \ref{prop:diss:not:cczc} that $\Gamma$ does not have conjugates in commuting $\Z$-conjugate.

    Suppose by contradiction that $\Gamma$ is binate. Consider $\Gamma_1 < \Gamma$, which is finitely generated and supported on $I_1$, and let $g \in \Gamma_1$ be an element of full support. By assumption, there exists a homomorphism $f \colon \Gamma_1 \to \Gamma$ such that $[\Gamma_1, f(\Gamma_1)] = 1$ and $t \in \Gamma$ such that ${}^t f(g) = g \cdot f(g)$. By Lemma \ref{lem:support:gamma}, there exist finitely many disjoint intervals $I_2, \ldots, I_n$ whose union equals the support of $f(g)$. Now, $\Gamma_1$ acts minimally on $I_1$, so $f(\Gamma_1)$ fixes $I_1$ pointwise; in particular $I_1$ is disjoint from $I_k$ for $2 \leq k \leq n$. It follows that $g \cdot f(g)$ is supported on the disjoint union $I_1 \cup \cdots \cup I_n$. This gives $t(I_2 \cup \cdots \cup I_n) = I_1 \cup \cdots \cup I_n$, a contradiction.
\end{proof}

\subsection{Correctness of Figure \ref{fig:vd-binate}}
\label{ss:vd-binate}

We need to exhibit the following examples:

\begin{enumerate}
    \item A binate group that does not have commuting conjugates (Proposition \ref{prop:binate:not:cc}).
    \item A group with commuting conjugates that is not binate (Example \ref{ex:F:reloaded}).
    \item A binate group with commuting conjugates that is not dissipated (Corollary \ref{cor:hall}).
\end{enumerate}

We leave the first item for the end, and start with the other two examples, which are well-known.

\begin{example}[c.\ c.\ not binate]
\label{ex:F:reloaded}
    We have already seen in Example \ref{ex:F} that $F'$ has conjugates in commuting $\Z$-conjugate but is not binate. In particular $F'$ is an example of a group with commuting conjugates that is not binate.
\end{example}

For the next example, we go a little further:

\begin{prop}
\label{prop:hall}
    Hall's universal locally finite group has the following combination of properties:
    \begin{enumerate}
        \item It is mitotic;
        \item It has commuting $\Z/n$-conjugates for every $n$;
        \item It does not have commuting $\Z$-conjugates.
    \end{enumerate}
\end{prop}

We immediately deduce:

\begin{cor}[Binate and c.\ c.\ not dissipated]
\label{cor:hall}
    Hall's universal locally finite group is binate, has commuting conjugates, but is not dissipated.
\end{cor}

\begin{proof}[Proof of Proposition \ref{prop:hall}]
    Recall that Hall's universal locally finite group $U$ is a countable locally finite group such that every finite group embeds into $U$, and such an embedding is unique up to conjugacy.\ It is easy to deduce from this universal property that $U$ is mitotic \cite{berrick}.
    
    Let $H \leq U$ be a finitely generated subgroup, which is therefore finite.\ By the first property, there exists an embedding $\iota\colon H \wr \Z/n \to U$.\ By the second property, there exists $s \in U$ such that $h = {}^s \iota(h)$ for all $h \in H$.\ Let $t_0$ denote the generator of $\Z/n$.\ Then $t \coloneqq {}^{s} \iota(t_0)$ satisfies the required properties.\ First $[H, {}^{t^p}H] = [{}^s \iota(H), {}^{{}^s \iota(t_0^p)} ({}^s \iota(H))] = {}^s \iota[H, {}^{t_0^p} H] = 1$ for all $1 \leq p < n$; and $t^n = {}^s \iota(t_0^n) = 1$, so clearly $[H, t^n] = 1$.\ This shows that $U$ has commuting $\Z/n$-conjugates for every $n$.

    On the other hand, $U$ does not have commuting $\Z$-conjugates because of being a non-abelian torsion group (Remark~\ref{rem:torsion:grps:no:cZc}).
\end{proof}

Finally, we construct a binate group that does not have commuting conjugates. In particular, this will be an example of a binate group that does not fit into either of the main categories of binate groups, namely mitotic and dissipated groups, since both of these have commuting conjugates (for different reasons).

We use a standard construction to embed groups into binate groups, called \emph{binate towers} \cite{berrick}. For a group $\Gamma$, we let
\[b(\Gamma) \coloneqq \langle \Gamma \times \Gamma, d \mid {}^d(1, g) = (g, g) \text{ for all } g \in \Gamma \rangle.\]

We can describe $b(\Gamma)$ as an HNN-extension with base $\Gamma \times \Gamma$ and associated subgroups $\Gamma_+ \coloneqq \{ 1 \} \times \Gamma$ and $\Delta_\Gamma \coloneqq \{ (g, g)\in\Gamma\times\Gamma \mid g \in \Gamma \}$.\ We denote by $\Gamma_- \coloneqq \Gamma \times \{ 1 \}$ and note that $\Gamma_-$ embeds into $b(\Gamma)$.\ This allows to iterate this construction by setting $b^0(\Gamma)\coloneqq \Gamma$ and $b^i(\Gamma) \coloneqq b(b^{i-1}(\Gamma))$ for all $i\geq 1$, and to construct the group $\beta(\Gamma)$ as the directed union of $\{ b^i(\Gamma) \}_{i \in \N}$.\ The group $\beta(\Gamma)$ is called the \emph{universal binate tower over $\Gamma$}, and it is the basic example of a binate group.\ We will show that any such construction gives the desired example:

\begin{prop}[Binate towers do not have c.\ c.]
\label{prop:binate:not:cc}
    For every finitely generated group $\Gamma$, the binate tower $\beta(\Gamma)$ does not have commuting conjugates.
\end{prop}

The key claim is the following structural result about the functor $b$.

\begin{lemma}
\label{lem:bassserre}
    Let $\Gamma$ be a group, identified with its image $\Gamma_- \leq \Gamma \times \Gamma \leq b(\Gamma)$. Then every element $1 \neq h \in \Gamma_-$ fixes a unique vertex in the Bass--Serre tree of $b(\Gamma)$, which is the vertex fixed by $\Gamma \times \Gamma$.
\end{lemma}

\begin{proof}
    Let $T$ be the Bass--Serre tree for the HNN-extension expression of $b(\Gamma)$.\ It has a unique orbit of vertices and exactly two orbits of oriented edges.\ For a suitable choice of base vertex $v$, incoming edge $e$ and outgoing edge $f$, the stabilizer of $v$ is $\Gamma \times \Gamma$, the stabilizer of $e$ is $\Delta_{\Gamma}$, and the stabilizer of $f$ is $\Gamma_+$.\ Moreover, all incoming edges are permuted by $\Gamma \times \Gamma$, therefore their stabilizers are conjugates of $\Delta_\Gamma$ in $\Gamma \times \Gamma$, and similarly for  outgoing edges and $\Gamma_+$.\ It follows that a non-trivial element $h \in \Gamma_-$ does not belong to the stabilizer of any edge adjacent to $v$, and therefore fixes no other vertex of $T$.
\end{proof}

\begin{proof}[Proof of Proposition \ref{prop:binate:not:cc}]
    Let $\Gamma$ be a finitely generated group.\ Up to replacing $\Gamma$ by $b(\Gamma)$, and noticing that $\beta(b(\Gamma)) \cong \beta(\Gamma)$, we may assume that $\Gamma$ is non-abelian.\ Suppose by contradiction that $\beta(\Gamma)$ has commuting conjugates.\ Then there exists $t \in \beta(\Gamma)$ such that $[\Gamma, {}^t \Gamma] = 1$, and by construction there exists $i \geq 0$ such that $t \in b^i(\Gamma)$.\ We choose $t$ so that $i$ is minimal.\ Then $i \geq 1$, for otherwise $t \in \Gamma$ and $[\Gamma, {}^t\Gamma] = [\Gamma, \Gamma] \neq 1$.

    We see $b^i(\Gamma)$ as $b(b^{i-1}(\Gamma))$ and notice that $\Gamma$ is identified with a subgroup of $b^{i-1}(\Gamma)_- \leq b^{i-1}(\Gamma) \times b^{i-1}(\Gamma) \leq b^i(\Gamma)$.\ By Lemma \ref{lem:bassserre}, if we denote by $v$ the vertex in the Bass--Serre tree that is fixed by $b^{i-1}(\Gamma) \times b^{i-1}(\Gamma)$, then $v$ is the unique vertex fixed by $\Gamma$.\ Therefore ${}^t \Gamma$ fixes the vertex $tv$ and no other vertex. On the other hand ${}^t \Gamma$ centralizes $\Gamma$, and therefore must preserve the fixed point set of $\Gamma$.\ It follows that $tv = v$, and so $t$ belongs to the stabilizer of $v$ in $b^i(\Gamma)$, that is $b^{i-1}(\Gamma) \times b^{i-1}(\Gamma)$, which allows to write $t = t_- t_+$ as a product of its projections onto the two factors.\ Now the conjugation by $t_+$ on $\Gamma \leq b^{i-1}(\Gamma)_-$ has no effect, and therefore $t_-$ satisfies $[\Gamma, {}^{t_-} \Gamma] = 1$.\ But $t_- \in b^{i-1}(\Gamma)_-$, which is the standard copy of $b^{i-1}(\Gamma)$ in $\beta(\Gamma)$.\ This contradicts the minimality of $i$ and concludes the proof.
\end{proof}

\subsection{Mitotic groups}
\label{ss:mitotic}

Mitotic groups form an important class of binate groups; However they are fundamentally different from the examples of boundedly supported groups with displacement properties that we have been mostly interested in. For example, one can show that if $\Gamma$ is a boundedly supported group acting on a space $X$ which exhibits commuting conjugates as in Example \ref{ex:cc:action}, and $\Gamma$ is moreover mitotic, then the support of every finitely generated subgroup of $\Gamma$ can be written as the disjoint union of an arbitrarily large number of open sets. However, we are unable to show that this situation cannot occur. In particular, the following problem is an obstacle for fitting mitotic groups into Figure \ref{fig:vd-binate}:

\begin{probl}[Dissipated and mitotic group]
\label{probl:mitotic:dissipated}
    Does there exist a dissipated group that is also mitotic?
\end{probl}

As for the interactions with other properties, it is clear that mitotic groups have commuting conjugates, and Hall's universal locally finite group witnesses that mitotic groups need not have commuting $\Z$-conjugates (Proposition \ref{prop:hall}). However, the following remains open:

\begin{probl}[Mitotic implies c.\ c.\ c.]
\label{probl:mitotic:ccc}
    Let $\Gamma$ be a mitotic group.\ Does $\Gamma$ have commuting cyclic conjugates?
\end{probl}

We propose a natural candidate for a counterexample to Problem \ref{probl:mitotic:ccc}.\
Following Baumslag--Dyer--Heller \cite[Section 5]{BDH}, for a group $\Gamma$, we denote by $m(\Gamma)$ its \emph{standard mitosis}, namely
\[m(\Gamma) \coloneqq \langle \Gamma \times \Gamma, s, d \mid {}^s (g, 1) = (1, g), {}^d (1, g) = (g, g) \rangle. \]
We can describe $m(\Gamma)$ as the fundamental group of a graph of groups whose underlying graph is a wedge of two circles, with unique vertex group $\Gamma \times \Gamma$ and with both edge groups $\Gamma$.\ Now $\Gamma$ embeds into $m(\Gamma)$ by first embedding into the first factor of $\Gamma \times \Gamma$.\ We can thus iterate this construction by setting $m^i(\Gamma) = m(m^{i-1}(\Gamma))$, and we denote by $\mu(\Gamma)$ the directed union of $\{ m^i(\Gamma) \}_{i \in \N}$, which is called the \emph{standard mitotic embedding} of $\Gamma$.\ Then $\mu(\Gamma)$ is mitotic \cite[Lemma 5.4]{BDH}, and thus binate.\
This construction can be seen as the mitotic analogue of the binate tower construction that we used for Proposition \ref{prop:binate:not:cc}.

\begin{probl}[Mitotic tower of free group has c.\ c.\ c.]
    Let $F_k$ denote a non-abelian free group.\ Does $\mu(F_k)$ have commuting cyclic conjugates?
\end{probl}

We suspect that it does not.\ At least, it is easy to show that $\mu(F_k)$ does not have commuting $\Z$-conjugates:

\begin{proof}
    If $\mu(F_k)$ has commuting $\Z$-conjugates, then there exists $t \in \mu(F_k)$ such that $[F_k, {}^{t^p} F_k] = 1$ for all $p \geq 1$.\ Since $F_k$ has trivial center, it follows that $\langle F_k, t \rangle$ contains a copy of $\oplus_{\N} F_k$, which has infinite cohomological dimension.\ Hence, if $i \in \N$ is such that $t \in m^i(F_k)$, it follows that $m^i(F_k)$ has infinite cohomological dimension as well.

    However it is easy to see that $cd(m(\Gamma)) = 2 cd(\Gamma)$ for every group $\Gamma$ \cite[Proposition 6.12]{bieri}. Therefore $cd(m^i(F_k)) = 2^i$, which contradicts the previous paragraph.
\end{proof}

Despite the similarity of the problem, adapting the argument from Proposition \ref{prop:binate:not:cc} to show that mitotic towers do not have commuting cyclic conjugates presents a challenge.\ Namely, a crucial step in the proof that $\beta(\Gamma)$ does not have commuting conjugates is the fact that the image $b^{i-1}(\Gamma) \to b^i(\Gamma)$ intersects trivially all edge stabilizers in the corresponding Bass--Serre tree.\ This is however not the case with $m^{i-1}(\Gamma) \to m^i(\Gamma)$, where the image is identified with an edge stabilizer, and in fact can be shown to fix infinitely many edges.

\section{Problems}\label{sec:problems}

We conclude this note with a list of intriguing problems about bounded acyclicity (see Section \ref{ss:mitotic} above for more problems).\ The first problem focuses on binate groups, that we know to be boundedly acyclic (Theorem~\ref{thm:binate:bac}), but we do not know whether they are also boundedly acyclic for all separable coefficients:

\begin{probl}[Boundedly not $\Xsep$-boundedly acyclic groups]
\label{probl:binate}
    Are binate groups $\Xsep$-boundedly acyclic? More generally, are there groups that are boundedly acyclic for trivial real coefficients, but that are not $\Xsep$-boundedly acyclic?\ If yes, how can we detect the non-vanishing of the bounded cohomology for certain separable coefficients?
\end{probl}

Due to Theorem \ref{thm:ccc:bac}, a natural candidate of a binate group that is not $\Xsep$-boundedly acyclic is the binate tower construction, since such groups do not have commuting cyclic conjugates:\ In fact they do not even have commuting conjugates (Proposition \ref{prop:binate:not:cc}).

\begin{probl}[$\Xsep$-bounded acyclicity of binate towers]
    Are universal binate towers $\beta(\Gamma)$ always $\Xsep$-boundedly acyclic?
\end{probl}

Another possible approach for constructing a boundedly acyclic group which is not $\Xsep$-boundedly acyclic goes via commensurability.\

\begin{lemma}
\label{lem:commensurability}
    If $\Lambda < \Gamma$ is a finite index subgroup, then $\Lambda$ is $\Xsep$-boundedly acyclic if and only if $\Gamma$ is $\Xsep$-boundedly acyclic.
\end{lemma}

\begin{proof}
    For every Banach $\Gamma$-module $E$, the restriction $H^n_b(\Gamma; E) \to H^n_b(\Lambda; E)$ is injective \cite[Proposition 8.6.2]{monod}.\ Therefore if $\Lambda$ is $\Xsep$-boundedly acyclic, then so is $\Gamma$.\

    Conversely, let $E$ be a dual separable $\Lambda$-module. Eckmann--Shapiro induction \cite[Proposition 10.1.3 (i)]{monod} yields a dual $\Gamma$-module $L^\infty(\Gamma, E)^{\Lambda}$ such that $H^n_b(\Lambda; E)\cong H^n_b(\Gamma; L^\infty(\Gamma, E)^{\Lambda})$ for all $n\geq 0$.\ 
    Let $R$ be a set of right coset representatives, so that $\Gamma = \Lambda R$. Then for all $\lambda \in \Lambda, r \in R, f \in L^\infty(\Gamma, E)^{\Lambda}$ it holds $f(\lambda r) = \lambda f(r)$~\cite[Notation 2.4.7]{monod}. So the map
    \[L^\infty(\Gamma, E)^{\Lambda} \to E^{[\Gamma : \Lambda]} \colon f \mapsto (f(r))_{r \in R}\]
    is an isomorphism of Banach spaces. Since $E$ is separable, this shows that $L^\infty(\Gamma, E)^{\Lambda}$ is separable. So if $\Gamma$ is $\Xsep$-boundedly acyclic, then 
    \[0=H^n_b(\Gamma; L^\infty(\Gamma, E)^\Lambda)\cong H^n_b(\Lambda; E)\]
    for all $n \geq 1$, which shows that $\Lambda$ is $\Xsep$-boundedly acyclic.
\end{proof}

The first direction also works with trivial coefficients, since the module does not change along the way: If $\Lambda$ is boundedly acyclic, then $\Gamma$ is boundedly acyclic. However, the other direction really needs the change of coefficients, which can be witnessed by explicit examples in degree $2$~\cite[Remark 2.4]{ccc}. Therefore, if $\Gamma$ is a boundedly acyclic group containing a finite index subgroup that is not boundedly acyclic, then $\Gamma$ is not $\Xsep$-boundedly acyclic.

\begin{probl}[Commensurability invariance of bounded acyclicity]
    Is bounded acyclicity invariant under commensurability?
\end{probl}

Another challenging problem is to understand whether the weakest notion of commuting conjugates is already enough to ensure bounded acyclicity.\ 
More precisely, it would be interesting to investigate the following:

\begin{probl}[c.\ c.\ versus bounded acyclicity]
    Let $\Gamma$ be a group with commuting conjugates.\ Is $\Gamma$ a boundedly acyclic group, or even an $\Xsep$-boundedly acyclic group? If the answer is negative, how can we show non-vanishing of the higher bounded cohomology groups, or of second bounded cohomology with non-trivial coefficients?
\end{probl}

Note that our examples of groups with commuting conjugates but without commuting cyclic conjugates are obtained as products of $\Xsep$-boundedly acyclic groups (Proposition \ref{prop:cc:not:ccc}); therefore they are also $\Xsep$-boundedly acyclic, as can be deduced from a spectral sequence argument with appropriate coefficients \cite[Proposition 12.2.2]{monod} (or variations of it \cite[Corollary 4.2.1]{moraschini_raptis_2023}).

\subsection*{Declarations} There is no conflict of interest to declare. There is no relevant data related to the paper.

\bibliographystyle{amsalpha}
\bibliography{svbib}

\end{document}